\documentclass[11pt]{article}
\usepackage[left=1in,right=1in,top=1in,bottom=1in]{geometry} 
\usepackage{graphicx} 

\usepackage[toc]{appendix}
\usepackage[normalem]{ulem} 
\usepackage{amsfonts,amsmath,amssymb,amsthm}

\usepackage{verbatim}
\usepackage{subcaption}
\usepackage{url}
 \usepackage{pgfplots}
 \pgfplotsset{compat=1.18}
\usepackage{hyperref}
\usepackage{orcidlink}
\usepackage{cleveref}
\usepackage[dvipsnames]{xcolor}

\usepackage[upint]{newtxmath} 

\definecolor{NiceBlue}{HTML}{0099E6}
\definecolor{NiceGreen}{HTML}{3BB300}
\definecolor{NiceGray}{HTML}{818589}
\definecolor{RED}{HTML}{FF0000}
\definecolor{DarkRed}{HTML}{8B0000}
\definecolor{VintageGreen}{HTML}{3C785B}
\definecolor{VintageRed}{HTML}{880B04}

\hypersetup{
	colorlinks=true,
	linkcolor=Sepia,
	filecolor=magenta,      
	urlcolor=Blue,
	citecolor=BlueViolet,
}
 
\newtheorem{theorem}{Theorem}
\newtheorem{lemma}[theorem]{Lemma} 
\newtheorem{proposition}[theorem]{Proposition} 
\newtheorem{remark}[theorem]{Remark}
\newtheorem{corollary}[theorem]{Corollary}

\theoremstyle{definition}

\newcommand{\EE}{\mathbb{E}}

\newcommand{\PP}{\mathbb{P}}
\newcommand{\RR}{\mathbb{R}}

\newcommand{\R}{\mathbb{R}}

\newcommand{\dint}{\mathop{}\!\mathrm{d}}

\newcommand{\1}{\mathbf{1}} 
\newcommand{\indic}[1]{\operatorname{\1}_{#1}}

\title{Wasserstein stability of the zero cell of a Poisson hyperplane tessellation under directional perturbations}
\author{Gilles Bonnet\,\orcidlink{0000-0003-4161-4124}
\and
Eliza O’Reilly\, \orcidlink{0000-0002-2646-292X}
\and
Bharath {Roy Choudhury}\, \orcidlink{0009-0008-9603-8631}
}
\date{}
\begin{document}

\maketitle

\begin{abstract}
A stationary Poisson hyperplane process in \(\mathbb{R}^d\) is characterized by an intensity parameter and an even probability measure on the unit sphere called the directional distribution.
In this work, we investigate the stability of the zero cell, i.e., the random convex polytope of the induced hyperplane tessellation containing the origin, under perturbations of the directional distribution. Our results provide quantitative bounds establishing local H\"{o}lder continuity of the distribution of the zero cell with respect to Wasserstein metrics on the space of random convex bodies and probability distributions on the unit sphere.
As an application, we establish stability bounds for density estimators constructed from Poisson hyperplane tessellations.
In particular, we bound the expected total variation distance between the zero-cell-based estimated probability measures in terms of the Wasserstein distance between their directional distributions.
\end{abstract}



\bigskip

\section{Introduction}

Random tessellations, or partitions of $\RR^d$ into convex polyhedral cells, appear in a wide variety of mathematical and statistical settings. In stochastic geometry, the random collection of cells is modeled as a point process on the space of convex bodies. One of the most studied random tessellation models in this literature is generated by a Poisson hyperplane process.  The distribution of a Poisson hyperplane process is determined by a locally finite measure $\Theta$, called the \emph{intensity measure}, on the product space
\(S^{d-1}\) $\times \mathbb{R}$ that parametrizes hyperplanes $H(u,t) := \{ x\in \RR^d : \langle x, u \rangle = t\}$.
A natural question arising from modeling applications is how the distribution of quantities related to the Poisson hyperplane process changes with respect to perturbations of this measure $\Theta$. 

In this work, we investigate this type of regularity in regard to the distribution of the \emph{zero cell} of the tessellation, i.e., the cell of the partition that contains the origin. This random convex polytope has been well-studied in the stationary setting, with many existing results on the distribution of its shape, size, and facial structure, and how these properties depend on the choice of intensity measure (see \cite[Chapter 6]{hugPoissonHyperplaneTessellations2024} and the references therein).
The underlying Poisson assumption on the hyperplanes also leads to a straightforward argument that the zero cell is weakly continuous with respect to the intensity measure (see Section \ref{sec:continuity}). However, this argument does not provide bounds that quantify how the zero cell changes under a perturbation of the intensity measure with respect to a metric on the space of convex bodies.

Motivation for this study comes from statistical learning applications. It has been observed that the geometry of this zero cell controls the bias when employing a Poisson hyperplane process (or a STIT tessellation \cite{Nagel2005}) to build a random partition estimator for nonparametric density estimation and regression \cite{OReillyTran2021, OReillyTran2021minimax}. When employing this model in this and other settings, one may want to estimate this intensity measure $\Theta$ using an available dataset to best fit the model to the underlying data distribution. To ensure stability of the model output with respect to errors in this estimation, one must ensure that the distribution of the tessellation does not change drastically with small perturbations of this parameter.

\subsection{Our Contributions and Outline}

This paper focuses on Poisson hyperplane tessellations that are \emph{stationary}, meaning that the distribution of the model is invariant under translations in $\RR^d$. 
In this setting, the intensity measure satisfies $\Theta = (\phi \otimes \gamma \lambda) \circ \mathfrak{q}^{-1}$, where $\lambda$ is the one-dimensional Lebesgue measure, \(\mathfrak{q}\) is defined by \eqref{eq:quotientHyperplaneMap}, $\gamma > 0$, and $\phi$ is an even probability measure on the unit sphere called a \emph{directional distribution}. 
We establish quantitative bounds that control how the distribution of the zero cell changes with respect to perturbations of the directional distribution.   

Such results require a choice of metric on the space of probability distributions on both the space of convex bodies and on the unit sphere. We consider the Wasserstein distance in both cases and we equip the space of convex bodies with various metrics, including the radial $p$-metric $\rho_p$ as defined in \eqref{eq:radial_metric} and the symmetric-difference volume metric $\Delta$ as defined in \eqref{e:symdiff_metric_def}.
The unit sphere is equipped with the natural geodesic metric ``geo'' defined in \eqref{eq:geo}.

Our main result (see Theorem \ref{thm:Wass_p_bnd_new}) is as follows.

\begin{theorem}
Let $Z_0^{(1)}$ and $Z_0^{(2)}$ be the zero cells of stationary Poisson hyperplane tessellations with intensity $\gamma$ and directional distributions $\phi_1$ and $\phi_2$, respectively. Then, for all $\phi_2$ close enough to $\phi_1$ in Wasserstein distance, 
\[W_{\operatorname{dist}}(Z_0^{(1)}, Z_0^{(2)}) 
\leq C_{\operatorname{dist}}(\phi_1) \, \gamma^{-\alpha(\operatorname{dist})} \, W_{\mathrm{geo}}(\phi_1,\phi_2)^{\beta(\operatorname{dist})} ,\]
where $\operatorname{dist}$ is either the radial metric $\rho_p$ with $p\in[1,\infty)$ or the symmetric-difference volume metric $\Delta$ and where $\alpha(\rho_p)=1$, $\beta(\rho_p) = \frac{1}{p}$, $\alpha(\Delta)=d$ and $\beta(\Delta)=\frac{1}{d}$.
\end{theorem}

This result says that the distribution of the zero cell is locally H\"{o}lder continuous with respect to the Wasserstein distance. In particular, it implies the following Wasserstein stability property for the zero cell: if $W_{\mathrm{geo}}(\phi_1,\phi_2)$ is small, then there is a coupling of the Poisson hyperplane processes such that $\EE[\operatorname{dist}(Z_0^{(1)}, Z_0^{(2)})]$ is small.

The proofs begin with the construction of a specific coupling of the zero cells of two Poisson hyperplane tessellations; see \Cref{s:coupling}. For this coupling, we derive the joint distribution of the radial functions in a fixed direction in \Cref{s:joint_distribution_radial_function}. We then use this joint distribution to obtain an expression for the tail distribution of the difference between the radial functions, which in turn yields upper bounds for the moments of this difference. The main Wasserstein stability results are proved in \Cref{sec:main_results}.

In Section \ref{sec:densityestimation}, we apply the Wasserstein stability results for the zero cell to density estimation with Poisson hyperplane partitions. 
Specifically, we construct zero-cell-based random probability densities whose distributions agree pointwise with those of the corresponding tessellation estimators, and we bound their expected total variation distance in terms of the Wasserstein distance between the directional distributions.

\section{Background and notation}\label{sec:background}

We now give a brief overview of the relevant background following the notation and definitions from \cite{hugPoissonHyperplaneTessellations2024} and \cite{MolchanovBook}.

\subsection{Random closed sets}

Let \(\mathcal{F}\) denote the space of closed subsets of \(\mathbb{R}^d\) endowed with the Effros sigma-algebra, i.e., the sigma-algebra generated by the collection of hits
\(\mathcal{F}_G:=\{F\in\mathcal{F}:F\cap G\neq\emptyset\}\) for all open sets \(G\subset\mathbb{R}^d\).
The Fell (hit-and-miss) topology on \(\mathcal{F}\) is the topology with subbasis consisting of the hits \(\{\mathcal{F}_G: G\text{ open}\}\) and the misses
\(\mathcal{F}^{C}:=\{F\in\mathcal{F}:F\cap C=\emptyset\}\) for all compact sets \(C \subset\mathbb{R}^d\).
Its Borel sigma-algebra coincides with the Effros sigma-algebra.
Since \(\mathbb{R}^d\) is LCHS (locally compact, Hausdorff, and second countable), the Effros sigma-algebra is the same as the sigma-algebra generated by the collection \(\mathcal{F}_C := \{F\in\mathcal{F}: F\cap C\neq\emptyset\}\) for all compact subsets \(C\subset\mathbb{R}^d\).

Let \(\mathcal{K}\) be the space of convex bodies, i.e., non-empty compact convex subsets of \(\mathbb{R}^d\), and let \(\mathcal{K}_0\) be the space of convex bodies containing the origin.
For \(K \in \mathcal{K}\), its \emph{support function} \(h_K:\mathbb{R}^d \to \mathbb{R}\) is defined for all \(x \in \mathbb{R}^d\) by
\begin{equation*}
    h(K,x) = \max\{\langle x,y \rangle: y\in K\}.
\end{equation*}
The Hausdorff metric \(d_H\) between any two convex bodies \(K,L \in \mathcal{K}\) is defined as
\begin{equation*}
    d_H(K,L) = \max\{\sup\{d(x,K): x \in L\}, \sup\{d(y,L): y \in K\}\},
\end{equation*}
where \(d(x,K) = \inf\{ \|x-z\|: z \in K\}\) and $\| \cdot \|$ is the usual Euclidean $L^2$ norm.
The Hausdorff metric can also be expressed in terms of support functions as
\begin{equation*}
    d_H(K,L) = \sup_{v \in S^{d-1}} |h(K,v) - h(L,v)|.
\end{equation*} 
On the one hand, \(\mathcal{K}\) inherits the subspace topology and the Effros sigma-algebra from \(\mathcal{F}\); on the other hand, the Hausdorff metric generates a Borel sigma-algebra on \(\mathcal{K}\).
These two sigma-algebras are one and the same \cite[Theorem 1.3.14 (iii)]{MolchanovBook}.
The collection of hits \(\mathcal{K}_G:=\{K \in \mathcal{K}:K \cap G \neq \emptyset\}\) for all open sets \(G \subset \mathbb{R}^d\) and the collection of misses \(\mathcal{K}^C:=\{K \in \mathcal{K}: K \cap C = \emptyset\}\) for all compact sets \(C \subset \mathbb{R}^d\) form the sub-basis for the subspace topology on \(\mathcal{K}\).

A random closed set is a measurable mapping $X$ from a probability space $(\Omega, \mathcal{A}, \mathbb{P})$ to the measurable space $(\mathcal{F}, \mathcal{B}(\mathcal{F}))$, where $\mathcal{F}$ is the space of closed subsets of $\mathbb{R}^d$ endowed with the Effros $\sigma$-algebra. When the image of $X$ is almost surely contained in $\mathcal{K}$, we refer to $X$ as a random convex body.
A random convex body of interest in this work, described in the following section, is the zero cell of a stationary Poisson hyperplane process with nondegenerate directional distribution.

\subsection{Hyperplane processes}

A hyperplane \(H=H(u,t)\) is a closed subset of \(\mathbb{R}^d\) defined by \(H(u,t):=\{x \in \mathbb{R}^d: \langle x, u \rangle =t \}\) for some $u \in S^{d-1}$ and $t \in \mathbb{R}$.
Note that \(H(u_1,t_1) = H(u_2,t_2)\) for \((u_1,t_1), (u_2,t_2) \in S^{d-1} \times \mathbb{R}\) if and only if \((u_1,t_1) = \pm (u_2,t_2)\).
For a hyperplane \(H(u,t)\), the closed half-space containing \(0\) is denoted by \(H^-(u,t)\).
Let \(\mathcal{H}\) denote the space of hyperplanes in \(\mathbb{R}^d\).
As a subspace of \(\mathcal{F}\), \(\mathcal{H}\) inherits the subspace topology and it is generated by the sub-basis of hits \([G]_{\mathcal{H}}:=\{H \in \mathcal{H}: H \cap G \neq \emptyset\}\) for all open sets \(G \subset \mathbb{R}^d\) and the misses \([C]_{\mathcal{H}}^c:=\{H \in \mathcal{H}: H \cap C = \emptyset\}\) for all compact sets \(C \subset \mathbb{R}^d\). 
The Borel sigma-algebra of the subspace topology on \(\mathcal{H}\) is also generated by the collection \([C]_{\mathcal{H}}:=\{H \in \mathcal{H}: H \cap C \neq \emptyset\}\) for all compact sets \(C\) in \(\mathbb{R}^d\).
Note that the map
\begin{align}\label{eq:quotientHyperplaneMap}
    \mathfrak{q}: S^{d-1} \times \mathbb{R} \to \mathcal{H}
\end{align}
sending \((u,t) \in S^{d-1} \times \mathbb{R}\) to \(H(u,t)\), where \(S^{d-1} \times \mathbb{R}\) is equipped with the product topology, induces the quotient topology on \(\mathcal{H}\).
The quotient and the subspace topologies on \(\mathcal{H}\) are one and the same and so are the sigma-algebras generated by them. 
Unless otherwise stated, measurability is understood with respect to this Borel sigma-algebra on \(\mathcal{H}\).

A measure \(\mu\) on \(\mathcal{H}\) is called \emph{locally finite} if \(\mu([C]_{\mathcal{H}})< \infty\) for all compact subsets \(C \subset \mathbb{R}^d\), and a \emph{counting measure} if \(\mu(B) \in \mathbb{N}\cup \{\infty\}\) for all measurable sets \(B \subset \mathcal{H}\).
Let \(\mathcal{N}:=\mathcal{N}(\mathcal{H})\) denote the measurable space of locally finite counting measures on \(\mathcal{H}\), equipped with the sigma-algebra generated by the evaluation maps \(E_B: \mathcal{N} \to \mathbb{N} \cup \{\infty\}\) for all Borel sets \(B\) in \(\mathcal{H}\), where \(E_B(\mu)=\mu(B)\) for every \(\mu \in \mathcal{N}\).

The space \(\mathcal{N}\) is endowed with the \emph{vague topology}, which is defined through \emph{vague convergence}.
A sequence \((\mu_n) \subset \mathcal{N}\) is said to converge vaguely to \(\mu \in \mathcal{N}\), denoted as $\mu_n \xrightarrow{v} \mu$, if
\begin{align} \label{eq:defVagueConvergence}
    \mu_n f := \int_{\mathcal{H}} f(H) \dint \mu_n(H) \to \int_{\mathcal{H}} f(H) \dint \mu(H):=\mu f
\end{align}
for all \(f \in C_c(\mathcal{H})\), the space of continuous and compactly supported functions on \(\mathcal{H}\), where continuity is understood with respect to the Fell topology (equivalently, the quotient topology).
The Borel sigma-algebra on \(\mathcal{N}\) generated by the open sets of the vague topology coincides with the sigma-algebra generated by the evaluation maps.
Equipped with the vague topology, \(\mathcal{N}\) is a Polish space (i.e., a complete separable metric space).

A \emph{hyperplane process} \(\Psi\) on \(\mathbb{R}^d\) is a measurable map from a probability space to \(\mathcal{N}\).
In other words, \(\Psi\) is a point process on \(\mathcal{H}\).
It is simple if \(\Psi(\{H\}) \in \{0,1\}\) for all \(H \in \mathcal{H}\).
Unless otherwise stated, all point processes considered here are simple, and therefore we identify them with their supports.
The intensity measure \(\Theta\) of \(\Psi\) is defined by \(\Theta(A) = \mathbb{E}[\Psi(A)]\) for all measurable sets \(A \subset \mathcal{H}\).
The vague topology defined by \cref{eq:defVagueConvergence} on \(\mathcal{N}\) extends to the space of locally finite measures on \(\mathcal{H}\).
We further assume that the hyperplane processes have locally finite intensity measures.
A Poisson hyperplane process on \(\mathbb{R}^d\) is a Poisson point process on $\mathcal{H}$ and is characterized by the intensity measure \(\Theta\) on \(\mathcal{H}\).

The group \(\mathbb{R}^d\) acts on \(\mathcal{H}\) by shifts: \(H(u,t) \mapsto H(u,t)+x = H(u, t+ \langle x,u \rangle)\) for all \(x \in \mathbb{R}^d\) and \(H(u,t) \in \mathcal{H}\).
For any simple locally finite counting measure \(\mu \in \mathcal{N}\) and \(x \in \mathbb{R}^d\), define \(\mu+x = \sum_{H \in \mathrm{supp}(\mu)} \delta_{H +x}\).
A hyperplane process \(\Psi\) on \(\mathbb{R}^d\) is called \emph{stationary} if the distribution of \(\Psi\) is invariant under all shifts, i.e., \(\Psi \overset{{(d)}}{=} \Psi+x\) for all \(x \in \mathbb{R}^d\).

A \emph{stationary Poisson hyperplane process} \(\Psi\) with locally finite intensity measure \(\Theta\) is characterized by a positive real number \(\gamma>0\) and an even probability measure \(\phi \in \mathcal{P}(S^{d-1})\) on \(S^{d-1}\) such that
\[\Theta(B)  
:= \mathbb{E}[\Psi(B)] 
= \gamma\int_{\RR} \int_{S^{d-1}} 1_{\{H(u,t) \in B\}} \dint \phi(u) \dint t
= 2 \gamma\int_{\RR_+} \int_{S^{d-1}} 1_{\{H(u,t) \in B\}} \dint \phi(u) \dint t , \]
for all measurable sets \(B\subset\mathcal{H}\) \cite[Chapter 4.4]{weil}.
The measure \(\phi\) is called the \emph{spherical directional distribution} associated with \(\Theta\) (or \(\Psi\)) and the constant $\gamma$ is called the \emph{intensity} of \(\Psi\).
Let \(\lambda_d\) denote the \emph{Lebesgue measure} on \(\mathbb{R}^d\) and let \(\lambda:=\lambda_1\); then \(\Theta = (\phi \otimes \gamma \lambda) \circ \mathfrak{q}^{-1}\), which is the pushforward of the product measure \(\phi \otimes \gamma \lambda\) under the map \(\mathfrak{q}\) (\cref{eq:quotientHyperplaneMap}).

An important parameter of a stationary Poisson hyperplane process \(\Psi\) is its \textit{associated zonoid} \cite[p.156]{weil}.
If \(\Psi\) has directional distribution $\phi$ and intensity $\gamma$, this is defined as the convex body $\Pi^{(\gamma)}$ in $\mathbb{R}^d$ with support function
\begin{align}\label{eq:hZ}
h(\Pi^{(\gamma)},v) = \frac{\gamma}{2}\int_{S^{d-1}} |\langle u,v \rangle| \dint \phi(u), \quad v \in S^{d-1}.
\end{align}
We will write $\Pi$ for the associated zonoid when $\gamma = 1$; this depends only on the directional distribution and will be called the \emph{normalized associated zonoid} of \(\Psi\).
 In particular, $\Pi^{(\gamma)} = \gamma \Pi$.

A hyperplane process \(\Psi\) induces a tessellation of \(\mathbb{R}^d\).
The \emph{zero cell} \(Z_0 := Z_0(\Psi)\) of a Poisson hyperplane process \(\Psi\) is the cell of the induced tessellation that contains the origin and it is a random closed set. If \(\Psi\) is stationary and the directional distribution \(\phi\) is nondegenerate, i.e., the support of \(\phi\) is not contained in any great subsphere of \(S^{d-1}\), then \(Z_0\) is a.s. bounded \cite[Chapter 4.3]{hugPoissonHyperplaneTessellations2024} and is thus a random convex body.
Unless stated otherwise, we assume throughout the paper that the directional distributions of all stationary Poisson hyperplane processes are nondegenerate.

\subsection{Wasserstein distance}

Let \((\chi,d_{\chi})\) be a Polish metric space.
The (extended) Wasserstein distance of (order \(1\)) between any two probability measures \(\mu, \nu\) on \(\chi\) is defined as 
\begin{align}
    \label{eq:Wasserstein_p}
    W_{d_{\chi}}(\mu,\nu) = \inf \biggl \{ \mathbb{E}[d_{\chi}(X,Y)] : \mathcal{L}(X,Y) \in \Pi(\mu,\nu) \biggr \},
\end{align}
where \(\Pi(\mu,\nu)\) denotes the set of couplings of \(\mu\) and \(\nu\), and the infimum is understood as an element of \([0,\infty]\). 
Since \((\chi,d_{\chi})\) is a Polish space, there exists a coupling of \((\mu,\nu)\) that attains the infimum in \eqref{eq:Wasserstein_p} if \(W_{d_{\chi}}(\mu,\nu)< \infty\) \cite[Theorem~4.1]{villaniOptimalTransportOld2009}.
Any such coupling is called an \textit{optimal coupling} of \((\mu,\nu)\).

Our aim in this work is to quantify how perturbations of the directional distribution \(\phi\) of a stationary Poisson hyperplane process affect the distribution of the zero cell \(Z_0\).
To this end, we introduce metrics on the space of inputs \(\mathcal{P}(S^{d-1})\) and on the space of outputs \(\mathcal{K}_0\) (the space of convex bodies containing \(0\)), together with their induced Wasserstein distances. 

First, for any two probability measures \(\phi_1,\phi_2\) on \(S^{d-1}\), let \(W_{\mathrm{geo}}(\phi_1,\phi_2)\) denote the Wasserstein distance induced by the geodesic distance \(\mathrm{geo}\) on \(S^{d-1}\), where
\begin{align}\label{eq:geo}
  \mathrm{geo}(u,v) = \arccos(\langle u, v \rangle)   
\end{align}
for all \(u,v \in S^{d-1}\).
Since \(S^{d-1}\) is compact, \(W_{\mathrm{geo}}(\phi_1,\phi_2)< \infty\).
Next, we define radial metrics and the symmetric-difference volume metric on \(\mathcal{K}_0\) and use their respective Wasserstein metrics on probability measures on \(\mathcal{K}_0\).

For any \(K \in \mathcal{K}_0\), the function \(\rho_K: S^{d-1} \to \mathbb{R}_{\geq 0}\) defined by
\begin{align}\label{eq:radialFunctionDef}
    \rho_K(u):= \sup\{t \geq 0: tu \in K\}
\end{align}
is called the \emph{radial function} of \(K\).
Let \(\sigma\) denote the \emph{spherical Lebesgue measure} on \(S^{d-1}\) defined by
\begin{align*}
    \sigma(A):= d \, \lambda_d\bigl(\{\alpha x : 0 \leq \alpha \leq 1, x \in A\}\bigr)
\end{align*}
for all Borel sets \(A \subset S^{d-1}\).
For \(p \in [1,\infty]\), the \emph{radial \(p\)-metric} between \(K,L \in \mathcal{K}_0\) is defined as 
\begin{align} \label{eq:radial_metric}
    d_{\rho_p}(K,L)=   
    \|\rho_K(\cdot) - \rho_L(\cdot) \|_p =
    	\begin{cases}
        	\left(\int_{S^{d-1}} |\rho_K(u) - \rho_L(u)|^p \sigma(\dint u)\right)^{1/p} & p<\infty, \\
            \sup_{u \in S^{d-1}} |\rho_{K}(u) - \rho_{L}(u)| & p=\infty.
        \end{cases}
\end{align}
The corresponding Wasserstein distance between any two random convex bodies \(Z_1,Z_2\) is denoted by \(W_{\rho_p}(Z_1,Z_2)\).

The \emph{symmetric-difference volume} distance \(d_{\Delta}\) between \(K,L \in \mathcal{K}\) is defined by 
\begin{align}\label{e:symdiff_metric_def}
    d_{\Delta}(K,L) = \lambda_d (K \Delta L) = \lambda_d (K \cup L) - \lambda_d (K \cap L).
\end{align}

If \(K,L \in \mathcal{K}_0\) , then \(d_{\Delta}\) can be expressed in terms of the radial function. In this case,
\begin{align*}
    \lambda_d(K \cup L) = \int_{K\cup L} \dint x
        = \int_{S^{d-1}} \int_0^{\rho_{K\cup L}(u)} t^{d-1} \dint t \sigma(\dint u)
        = \frac{1}{d}\int_{S^{d-1}} \rho_{K\cup L}(u)^d \sigma(\dint u),
\end{align*}
and 
\begin{align*}
    \rho_{K\cup L}(u)^d - \rho_{K\cap L}(u)^d
    &= \max (\rho_{K}(u)^d , \rho_{L}(u)^d) - \min (\rho_{K}(u)^d , \rho_{L}(u)^d)
    = | \rho_{K}(u)^d - \rho_{L}(u)^d |.
\end{align*}
Therefore, for \(K,L \in \mathcal{K}_0\), we obtain 
\begin{align}
    \notag
    d_{\Delta}(K,L) 
    &= \lambda_d (K \cup L) - \lambda_d (K \cap L) = \frac{1}{d} \int_{S^{d-1}} \bigl(\rho_{K\cup L}(u)^d - \rho_{K\cap L}(u)^d\bigr) \sigma(\dint u) 
    \\ \label{def:symmDiff}
    &= \frac{1}{d} \int_{S^{d-1}} |\rho_K(u)^d - \rho_L(u)^d| \sigma(\dint u).
\end{align}
The corresponding Wasserstein distance between any two random convex bodies \(Z_1,Z_2\) \(\in \mathcal{K}_0\) is denoted by \(W_{\Delta}(Z_1,Z_2)\).

\section{Continuity of the zero cell}\label{sec:continuity}

In this section, we give detailed arguments for the intuitive claim that the zero cell of a Poisson hyperplane tessellation is continuous with respect to the intensity measure of the Poisson hyperplane process. 

The first result establishes weak convergence of the zero cells under vague convergence of a general class of nondegenerate intensity measures. 
The second proves the converse for stationary Poisson hyperplane processes. 
Finally, the third strengthens the mode of convergence in the stationary setting: when the processes have a common intensity parameter and their directional distributions converge weakly, the zero cells can be coupled to converge almost surely in the Hausdorff metric.

\subsection{Weak continuity}

We first establish the following continuity statement for zero cells of Poisson
hyperplane processes: whenever the associated intensity measures converge vaguely, the laws
of the zero cells converge weakly.
We then proceed to show that the converse also holds in the stationary case.

We use the following notation in this section.
For any compact set \(K \subset \mathbb{R}^d\), denote the convex hull of \(K\) and \(\{0\}\) by \(K_0\), \([K]_{\mathcal{H}}:=\{H \in \mathcal{H}: H \cap K \neq \emptyset\}\), and \(\operatorname{int} [K]_{\mathcal{H}}\) by the interior of \([K]_{\mathcal{H}}\) in \(\mathcal{H}\).
Note that for any compact convex body \(K\), the set \([K]_{\mathcal{H}}\) is closed in \(\mathcal{H}\).
Therefore, the set \([K]_{\mathcal{H}}^c\) of hyperplanes that avoid \(K\), is a sub-basis element and hence it is an open set.
We will also use the notation \([x]_{\mathcal{H}}\) instead of \([\{x\}]_{\mathcal{H}}\) to denote the set of hyperplanes passing through a point \(x\).
Let \(\mathcal{H}_*:= \mathcal{H} \backslash [0]_{\mathcal{H}}\) denote the set of hyperplanes that avoid \(0\). 
A hyperplane \(H\) is said to \emph{strictly separate} two points, if the points lie in the opposite open half-spaces of \(H\).

\begin{proposition}
    \label{prop:convViaInclusion}

    Let \((\Theta_n)_{n\in\mathbb{N}}\) and \(\Theta\) be locally finite measures on the space of hyperplanes \(\mathcal{H}\) that satisfy \(\Theta([0]_{\mathcal{H}})=0\) and \(\Theta_n([0]_{\mathcal{H}})=0\) for all \(n\).
    Let \(Z_0^n\) and \(Z_0\) be the zero cells of the Poisson hyperplane processes \(\Psi_n\) and \(\Psi\) whose intensity measures are \(\Theta_n\) and \(\Theta\) respectively.
    Assume that \(Z_0^n\) and \(Z_0\) have non-empty interior and are compact almost surely. 
    If \(\Theta_n\) converges vaguely to \(\Theta\), then \(Z_0^n\) converges to \(Z_0\) in distribution.
\end{proposition}
The condition in the above proposition that the intensity measures assign measure \(0\) to \([0]_{\mathcal{H}}\) ensures that almost surely no hyperplane passes through the origin, so that the closed half-space bounded by each hyperplane and containing the origin is unambiguously defined.
The proof of \Cref{prop:convViaInclusion} (below) relies on the convergence criterion \cite[Proposition 1.8.16]{MolchanovBook}, which states, under the assumption that all the random convex bodies have non-empty interior, that a sequence of random convex bodies \(Z_n\) converges to a random convex body \(Z\) if and only if \(\mathbb{P}[K \subset Z_n] \to \mathbb{P}[K \subset Z]\) as \(n \to \infty\) for all convex compact sets \(K\) such that \(\mathbb{P}[K \subset Z] = \mathbb{P}[K \subset \operatorname{int} Z]\).
We will also use the following characterization for vague convergence of measures on \(\mathcal{H}\) (see \cite[Lemma 4.1]{kallenbergRandomMeasuresTheory2017}): a sequence \((\mu_n)\) of locally finite measures on \(\mathcal{H}\) converges vaguely to a locally finite measure \(\mu\) if and only if \(\mu_n (A) \to \mu(A)\) for all continuity sets \(A\) of \(\mu\), i.e., subsets \(A\) such that the boundary of \(A\) has measure \(0\) with respect to \(\mu\).

\begin{proof}[Proof of \Cref{prop:convViaInclusion}]

Since the intensity measures assign measure \(0\) to \([0]_{\mathcal{H}}\), without loss of generality we assume that the hyperplanes considered in the proof belong to \(\mathcal{H}_*\) unless stated otherwise.
    Let \(K\) be a non-empty compact convex subset of \(\mathbb{R}^d\) and \(K_0\) be the convex hull of \(K \cup \{0\}\).
 Consider the set
 \begin{align*}
    S_K:= \{H \in \mathcal{H}_*: H \text{ strictly separates }0 \text{ and some point of }K \}.
 \end{align*}
 Then, for any hyperplane \(H(u,t)\), the following equivalent conditions hold:
 \begin{align*}
    H(u,t) \in S_K &\iff \min_{x \in K_0} \langle x, u \rangle < t < \max_{x \in K_0} \langle x, u \rangle, \\
    H(u,t) \in [K_0]_{\mathcal{H}} &\iff \min_{x \in K_0} \langle x, u \rangle \leq t \leq \max_{x \in K_0} \langle x, u \rangle.
 \end{align*}
 Therefore, \(S_K = \text{int} [K_0]_{\mathcal{H}}\).

 Observe that \(K \subset Z_0\) if and only if no hyperplane of \(\Psi\) belongs to \(S_K\), equivalently \(\Psi(\text{int}[K_0]_{\mathcal{H}})=0\), and similarly, \(K \subset \text{int} Z_0\) if and only if no hyperplane of \(\Psi\) belongs to \([K_0]_{\mathcal{H}}\).
Similar statements can be made for $Z_0^n$ and $\operatorname{int}(Z_0^n)$, and because our hyperplane processes are Poisson, this leads to the following equalities:

\begin{align*}
        \PP[K \subset Z_0]
        &= \exp\left( - \Theta( \operatorname{int} [K_0]_{\mathcal{H}} ) \right),\\
        \PP[K \subset Z_0^n] 
        &= \exp\left( - \Theta_n( \operatorname{int} [K_0]_{\mathcal{H}} ) \right), \\
       \text{ and } \PP[K \subset \operatorname{int} Z_0] 
        &= \exp\left( - \Theta( [K_0]_{\mathcal{H}} ) \right).
    \end{align*}
Now, consider a convex body for which   
\(\PP[K \subset Z_0] = \PP[K \subset \operatorname{int} Z_0]\).
   By the above equalities this means that
    $\Theta(  \operatorname{int} [K_0]_{\mathcal{H}} ) = \Theta( [K_0]_{\mathcal{H}} )$,
  equivalently,
  
\[ 
    \Theta(\partial \operatorname{int} [K_0]_{\mathcal{H}})
    = \Theta\bigl(\overline{\operatorname{int}[K_0]_{\mathcal{H}}} \setminus  \operatorname{int}[K_0]_{\mathcal{H}} \bigr) \leq \Theta \bigl([K_0]_{\mathcal{H}} \setminus \operatorname{int} [K_0]_{\mathcal{H}} \bigr) = 0 ,\]

    where the inequality in the above follows since \([K_0]_{\mathcal{H}}\) is a closed subset of \(\mathcal{H}\).
    Thus  \(\operatorname{int} [K_0]_{\mathcal{H}}\) is a continuity set of $\Theta$, and 
     
  \[
    \Theta_n(\operatorname{int} [K_0]_{\mathcal{H}}) \to \Theta(\operatorname{int} [K_0]_{\mathcal{H}}),
    \]
    because $\Theta_n\to\Theta$ vaguely.
    Therefore
    
    \begin{align*}
        \PP[K \subset Z_0^n] 
        &\to \PP[K \subset Z_0] .
    \end{align*}
    We can then conclude by \cite[Proposition 1.8.16]{MolchanovBook} that $Z_0^n \to Z_0$ in distribution.
\end{proof}

If the Poisson hyperplane processes in \Cref{prop:convViaInclusion} are stationary, the next proposition shows that vague convergence of the intensity measures is equivalent to weak convergence of the zero cells.
Note that the assumptions stated in \Cref{prop:convViaInclusion} are satisfied by all stationary Poisson hyperplane processes.

For a finite even signed measure \(\mu\) on \(S^{d-1}\), its cosine transform \(C_{\mu}:S^{d-1} \to \mathbb{R}\) is defined by
\begin{align*}
    C_{\mu}(v):= \int_{S^{d-1}} | \langle v,u \rangle| \, \mu(du), v \in S^{d-1}.
\end{align*}
The cosine transform uniquely determines an even finite signed measure.
That is, if \(\mu_1,\mu_2\) are two even finite measures on \(S^{d-1}\) such that \(C_{\mu_1}(v) = C_{\mu_2}(v)\) for all \(v \in S^{d-1}\), then \(\mu_1 = \mu_2\) \cite[Lemma 5.2.2]{hugPoissonHyperplaneTessellations2024}.

\begin{proposition} \label{prop:zero-cell-continuity}
Let \(\phi_n,\phi\in\mathcal{P}(S^{d-1})\) be even probability measures, and let \(\gamma_n,\gamma > 0\).
For each \(n\), let \(Z_0^n\) be the zero cell of the stationary Poisson hyperplane process \(\Psi_n\) with intensity measure
\(\Theta_n=(\phi_n \otimes \gamma_n \lambda)\circ \mathfrak{q}^{-1}\) and let \(Z_0\) be the zero cell of the stationary Poisson hyperplane process \(\Psi\) with intensity measure \(\Theta=(\phi \otimes \gamma \lambda)\circ \mathfrak{q}^{-1}\).
Then, \(\phi_n \Rightarrow \phi\) (weakly) and \(\gamma_n \to \gamma\) if and only if \(Z_0^n\) converges to \(Z_0\) in distribution as random convex bodies as $n \to \infty$.
\end{proposition}
\begin{proof}
    First observe that if \(\gamma_n \to \gamma\) and \(\phi_n \Rightarrow \phi\) then \(\Theta_n \overset{v}{\to} \Theta\) and therefore by \Cref{prop:convViaInclusion},  \(Z_0^n\) converges in distribution to \(Z_0\).
    Conversely, assume that \(Z_0^n\) converges in distribution to \(Z_0\).
    Fix \(x \in \mathbb{R}^d\), and let 
    \[A_x:=\{H \in \mathcal{H}: H \text{ strictly separates } 0 \text{ and } x\}.\]

    Note that \(x \in \operatorname{int} Z_0^n\) if and only if no hyperplane of \(\Psi_n\) strictly separates \(x\) and \(0\). 
   Moreover, in the stationary case the set of hyperplanes passing through a fixed point \(x\) has \(\Theta_n\)-measure zero. 
   Hence 
   \[
    \mathbb{P}[x \in \partial Z_0^n] 
    \leq 
    \mathbb{P}[\Psi_n([x]_{\mathcal{H}})>0]
    =
    1 - \exp(- \Theta_n([x]_{\mathcal{H}}))
    =0,
    \]
   and therefore 
    \begin{align*}
        \mathbb{P}[x \in \operatorname{int}Z_0^n] = \mathbb{P}[x \in Z_0^n] = \mathbb{P}[\Psi_n(A_x) = 0] = \exp(-\Theta_n(A_x)).
    \end{align*}
Similarly,
    \begin{align*}
        \mathbb{P}[x \in \operatorname{int}Z_0] = \mathbb{P}[x \in Z_0] = \exp(-\Theta(A_x)).
    \end{align*}
    Since \(Z_0^n\) converges in distribution to \(Z_0\), Proposition 1.8.16 in \cite{MolchanovBook} (also described after the statement of Proposition \ref{prop:convViaInclusion} above) implies that \(\mathbb{P}[x \in \operatorname{int}Z_0^n] \to \mathbb{P}[x \in \operatorname{int}Z_0]\).
    Hence \(\Theta_n(A_x) \to \Theta(A_x)\).

  Now, by the stationary representation of the intensity measure, for every \(x\in\mathbb R^d\), as $n \to \infty$,
    \begin{align}\label{eq:convergenceCosineTransform}
        \Theta_n(A_x) &= \mathbb{E}\bigg[\int_{\mathcal{H}} \mathbf{1}_{A_x}(H) \Psi_n(d H)\bigg] = \gamma_n \int_{S^{d-1}} |\langle x,u \rangle| \phi_n(du) \to  \Theta(A_x) =\gamma \int_{S^{d-1}} |\langle x,u \rangle| \phi(du).
    \end{align}
Then \eqref{eq:convergenceCosineTransform} implies that

\begin{align} \label{eq:2convergenceCosineTransform}
 \gamma_n C_{\phi_n}(v) \to \gamma C_{\phi}(v),
 \qquad v \in S^{d-1}.
\end{align}

First, we show that \(\gamma_n \to \gamma\) as \(n \to \infty\).
 Consider the set \([B_d]_{\mathcal{H}}\) of hyperplanes intersecting the closed unit ball in \(\mathbb{R}^d\).
 Then,
 \begin{align*}
    \Theta\bigl([B_d]_{\mathcal{H}} \bigr)
    &=
    \gamma \int_{\mathbb{R}} \int_{S^{d-1}} 
    \mathbf{1}\{H(u,t) \cap B_d \neq \emptyset\} \phi (du) dt
    = \gamma \int_{S^{d-1}} \lambda([-1,1]) \phi(du) = 2 \gamma.
 \end{align*}
 Similarly, \(\Theta_n\bigl([B_d]_{\mathcal{H}}\bigr) = 2 \gamma_n\).
 Since the set \(T\) of hyperplanes that are tangent to \(S^{d-1}\) have \(\Theta\)-measure \(0\), we have
 \begin{align*}
    \mathbb{P}[\{B_d \subset Z_0\} \backslash \{B_d \subset \operatorname{int} Z_0\}]
    &=
    \mathbb{P}[B_d \subset Z_0, B_d \cap \partial Z_0 \neq \emptyset]
    \leq \mathbb{P}[\Psi(T)>0]\\
    &=1-\exp(-\Theta(T))=0.
 \end{align*}
 Similar results hold for \(Z_0^n\).
 Thus, \(B_d\) is a continuity set for \(Z_0\) and for \(Z_0^n\) for all \(n\).

 Since \(Z_0^n\) converges to \(Z_0\) in distribution, by Proposition 1.8.16 in \cite{MolchanovBook},
 \[
 \mathbb{P}[B_d \subset \operatorname{int} Z_0^n] 
 = 
 \mathbb{P}[B_d \subset Z_0^n] 
 \to 
 \mathbb{P}[B_d \subset Z_0] 
 = \mathbb{P}[B_d \subset \operatorname{int} Z_0].
 \]
 Note that
 \[
 \mathbb{P}[B_d \subset \operatorname{int}Z_0^n] 
 = 
 \mathbb{P}[\Psi_n([B_d]_{\mathcal{H}})=0]
 =
 \exp(-\Theta_n([B_d]_{\mathcal{H}}))
 = e^{-2 \gamma_n},
 \]
 and similarly \(\mathbb{P}[B_d \subset \operatorname{int}Z_0] = e^{-2 \gamma}\).
 Therefore, \(\gamma_n \to \gamma\) as \(n \to \infty\).

 We now show that \(\phi_n \Rightarrow \phi\).
 Since \(\gamma_n \to \gamma\), from \Cref{eq:2convergenceCosineTransform} we have \(C_{\phi_n}(v) \to C_{\phi}(v)\) for all \(v \in S^{d-1}\). 

Since \((\phi_n)\) is a sequence of probability measures on the compact space \(S^{d-1}\), it follows that every subsequence of \((\phi_n)\) has a further subsequence that converges weakly to a probability measure on \(S^{d-1}\). 
Since each \(\phi_n\) is even, every such weak limit is also an even probability measure.

Let \((\phi_{n_k})\) be a subsequence converging weakly to an even probability measure \(\phi'\) on \(S^{d-1}\).
Since the map \(u \mapsto |\langle v,u \rangle|\) is continuous and bounded on \(S^{d-1}\) for every \(v \in S^{d-1}\), weak convergence gives \(C_{\phi_{n_k}}(v) \to C_{\phi'}(v), \; v \in S^{d-1}\).
On the other hand, by \eqref{eq:2convergenceCosineTransform}, \(C_{\phi_{n_k}}(v) \to C_{\phi}(v)\) for all \(v \in S^{d-1}\).
Hence \(C_{\phi'}(v)=C_\phi(v)\) for all \(v\in S^{d-1}\), and therefore \(\phi'=\phi\) by the remark preceding the proposition.
Thus, every weakly convergent subsequence of \((\phi_n)\) has limit \(\phi\), and therefore \(\phi_n\Rightarrow\phi\).

This completes the proof.
\end{proof}

\subsection{Convergence in Hausdorff metric}
If we assume the Poisson hyperplane processes are stationary with the same intensity parameter $\gamma$, we can obtain a stronger continuity result. In this case, we can show almost sure convergence with respect to the Hausdorff distance of the zero cells as the directional distributions converge weakly.

Recall that all the directional distributions considered are nondegenerate, i.e., their supports are not contained in any great subsphere.

\begin{proposition}
    Let $(\phi_n)_{n \geq 0}$ be a sequence of directional distributions such that \(\phi_n \Rightarrow \phi_0\) as \(n \to \infty\).
    Let $\gamma>0$.
    Then there exists a sequence $(Z_0^n)_{n \geq 0}$ of random polytopes such that
    \begin{itemize}
        \item[(i)] $Z_0^n$ is distributed as the zero cell of a stationary Poisson hyperplane process with directional distribution $\phi_n$ and intensity $\gamma$, for any $n \geq 0$, and  
        \item[(ii)] $d_H(Z_0^n,Z_0^0) \to 0$ as $n\to\infty$, almost surely.
    \end{itemize}
\end{proposition}
\begin{proof}
    First, observe that by Skorokhod's representation theorem, there exists a random sequence of unit vectors $U=(U^0, U^1, \ldots)$ such that $U^n$ has distribution $\phi_n$ for any $n \geq 0$, and $U^n \to U^0$ as $n \to \infty$, almost surely.

    Let $(U_j)_{j\geq 1}$ be a collection of i.i.d.\ copies of $U$.
    Let $T=(T_1<T_2<\cdots)$ be a homogeneous Poisson process on $(0,\infty)$ of intensity $2\gamma$ that is independent of \((U_j)_{j \geq 1}\).
    For any $n \geq 0$, using the standard representation of a stationary hyperplane process by distances \(T_j>0\) from the origin and i.i.d. directions \(U_j^n\sim \phi_n\),
    the collection $\{ H(U_j^n,T_j) : j \in\mathbb{N} \}$ is a stationary Poisson hyperplane process with directional distribution $\phi_n$ and intensity $\gamma$.
    For all $n \geq 0$, let 
    \begin{align}\label{e:zerocelln_rep}
    Z_0^n = \bigcap_{j\in\mathbb{N}} H^-(U_j^n,T_j)
    \end{align}
    be the corresponding zero cells.
    We will show that the statement holds for $(Z_0^n)_{n \geq 0}$.

    Let $R' = \max \{ \|x\| : x \in  Z_0^0\}$, \(R:= R'+1\),  $M = \max \{ j : T_j < R\} $, and \(B_R\) be the closed ball centered at \(0\) and of radius \(R\).
    We can thus write
    \begin{align}
        \label{eq:2926a}
        Z_0^0 = \bigcap_{j=1}^M H^-(U_j^0,T_j) .
    \end{align} 
    We can bound $Z_0^n$ by sets comparable to the last intersection.
    Namely we have
    \begin{align}
        \label{eq:2926b}
        \left(\bigcap_{j=1}^M H^-(U_j^n,T_j)\right) \cap B_R
        = \left(\bigcap_{j=1}^\infty H^-(U_j^n,T_j) \right) \cap B_R
        \subset Z_0^n 
        \subset \bigcap_{j=1}^M H^-(U_j^n,T_j).
    \end{align}
    \begin{figure}
        \centering
        \includegraphics[width=0.5\linewidth]{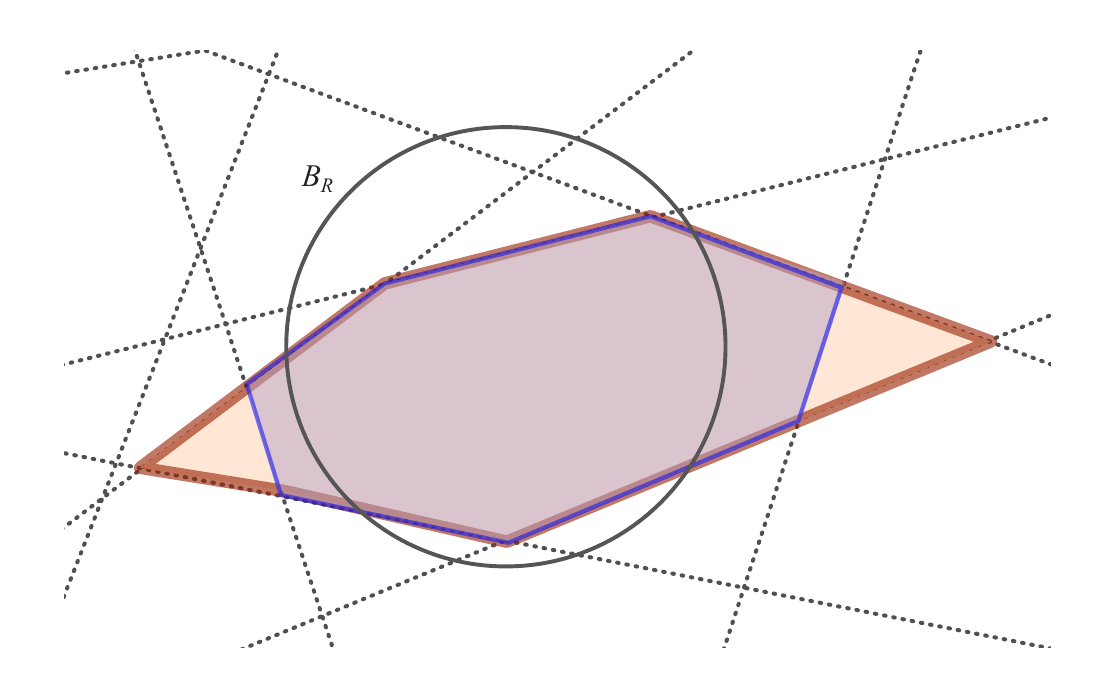}
        \caption{Illustration of the inclusions in \Cref{eq:2926b}.
        The brown and blue polygons represent $\bigcap_{j=1}^M H^-(U_j^n,T_j)$ and $\bigcap_{j=1}^\infty H^-(U_j^n,T_j) = Z_0^n$, respectively.}
        \label{fig:prop5}
    \end{figure}%
    The equality follows from the fact that $B_R \subset H^-(u,T_j)$ for any $u\in S^{d-1}$ and any $j>M$ because $T_j\geq R$ for such $j$,
    and both inclusions follow from the representation of $Z_0^n$ in \eqref{e:zerocelln_rep}. See \Cref{fig:prop5} for an illustration of these inclusions.

    Now, consider arbitrary $0 < t_1 < \cdots < t_{m} < \infty$, $r\in(0,\infty)$ and $u_1^0,\ldots,u_{m}^0 \in S^{d-1}$ such that $\cap_{j=1}^{m} H^-(u_j^0 , t_j) \subset B_r$.
    Additionally, consider unit vectors $(u_j^n)_{1\leq n ; 1\leq j \leq m}$.
    Observe that if $u_j^n \to u_j^0$ as $n \to\infty$, for any $1\leq j \leq m$, then the intersection $\cap_{j=1}^{m}  H^-(u_j^n , t_j)$, 
    converges to $\cap_{j=1}^{m} H^-(u_j^0 , t_j)$ in Hausdorff distance, as $n \to\infty$. 
    
    Applying this fact to the random elements above and using the almost sure convergences $U^n_j \to U^0_j$, we get that the intersections that appear on both sides of \eqref{eq:2926b} converge (almost surely, in Hausdorff distance) to the intersection on the right-hand side of \eqref{eq:2926a}, as the intersection on the right-hand side of \eqref{eq:2926b} is bounded for large \(n\).
    It follows that \(d_H(Z_0^n,Z_0^0)\to0\) almost surely.
\end{proof}
 
\section{Coupling of stationary Poisson hyperplane processes}\label{s:coupling}

In the preceding section, we established qualitative continuity properties of the zero cell. 
The aim of this and the following section is to seek quantitative stability estimates for the distributions of the zero cells in Wasserstein distance.
To this end, we introduce an explicit coupling of any two stationary Poisson hyperplane processes through a specific coupling of their directional distributions.
This construction allows us to compare the corresponding zero cells on a common probability space and to analyze the joint behaviour of their radial functions.
In particular, we derive formulas for the joint survival function of the coupled radial functions, for the tail distribution of their difference, and for the corresponding moments.

These estimates culminate in \Cref{thm:Wass_p_bnd_new} (\Cref{sec:main_results}), which establishes local Hölder continuity of the zero-cell distribution with respect to perturbations of the directional distribution. 
In particular, the result not only yields convergence when the directional distributions approach one another, but also provides an explicit rate of convergence.

\subsection{Notation and coupling of zero cells} \label{sec:notation_coupling}

Let \((S^{d-1})^2:=S^{d-1}\times S^{d-1}\).
For \(i=1,2\), define the projection maps
\[
p_i : (S^{d-1})^2 \times \mathbb{R} \to S^{d-1}\times\mathbb{R},
\qquad
p_i((u_1,u_2),t)=(u_i,t).
\]

We use the notation \((u_1,u_2,t)\) for an element \(((u_1,u_2),t) \in (S^{d-1})^2 \times \mathbb{R}\).
Let \(\phi_1,\phi_2 \in \mathcal{P}(S^{d-1})\) be nondegenerate even directional distributions, and let
\(\Psi_1\) and \(\Psi_2\) be the stationary Poisson hyperplane processes with common intensity parameter \(\gamma>0\) and directional measures \(\phi_1,\, \phi_2\), respectively.
A coupling of \(\Psi_1\) and \(\Psi_2\) is constructed as follows.

Choose a coupling \(\Phi \in \Pi(\phi_1,\phi_2)\) of the directional distributions, that is, a probability measure
\(\Phi\) on \((S^{d-1})^2\) whose marginals are \(\phi_1\) and \(\phi_2\).
Let \(\Xi\) be the Poisson process on \((S^{d-1})^2 \times \mathbb{R}\) with intensity measure
\(\Phi \otimes \gamma\lambda\).
For \(i=1,2\), define the random counting measure \(\tilde{\Psi}_i\) on \(\mathcal{H}\) by
\[
\tilde{\Psi}_i := \Xi \circ (\mathfrak q\circ p_i)^{-1}.
\]
Equivalently,
\begin{align}
    \tilde{\Psi}_i=\sum_{((u_1,u_2),t)\in \Xi}\delta_{H(u_i,t)}, \qquad i=1,2.
\end{align}
Then \(\tilde{\Psi}_1\) and \(\tilde{\Psi}_2\) are stationary Poisson hyperplane processes with directional distributions
\(\phi_1\) and \(\phi_2\), respectively, and both have intensity parameter \(\gamma\).
Thus, \((\tilde{\Psi}_1,\tilde{\Psi}_2)\) is a coupling of \(\Psi_1\) and \(\Psi_2\).

The coupling \((\tilde{\Psi}_1,\tilde{\Psi}_2)\)  associated with
\(\Phi \in \Pi(\phi_1,\phi_2)\), induces a coupling of the corresponding zero cells.
More precisely, \(\bigl(Z_0(\tilde{\Psi}_1),\,Z_0(\tilde{\Psi}_2)\bigr)\) is a coupling of \(Z_0(\Psi_1)\) and \(Z_0(\Psi_2)\).
To simplify the notation, throughout the remainder of the paper we write
\begin{align}\label{eq:couplingZeroCells}
    Z_0^{(1)}:= Z_0(\tilde{\Psi}_1) \qquad \text{and}  \qquad Z_0^{(2)}:= Z_0(\tilde{\Psi}_2).
\end{align}

In general, a coupling of zero cells need not arise from a coupling of the underlying hyperplane processes in this way.
Nevertheless, throughout this paper we consider only couplings of zero cells induced by couplings of hyperplane processes constructed as above.

\subsection{Joint survival function of coupled radial functions}\label{s:joint_distribution_radial_function}\label{s:radial_coupling}

The purpose of this section is to study the joint survival function of the radial functions of the coupled zero cells $Z_0^{(1)}$ and $Z_0^{(2)}$ as defined in \eqref{eq:couplingZeroCells} in a fixed direction $u \in S^{d-1}$, where the coupling of the zero cells is induced by any fixed coupling \(\Phi \in \Pi(\phi_1,\phi_2)\).
This joint survival function plays a crucial role in deriving upper bounds for Wasserstein metrics.
Recall the definition of the radial function of a convex body from \eqref{eq:radialFunctionDef}.

\begin{proposition}
    \label{prop:jointSurvival}
    Given the coupling $(Z_0^{(1)},Z_0^{(2)})$ defined by \eqref{eq:couplingZeroCells},
    we have, for all $r_1, r_2 > 0$,
    \begin{align*} 
        \mathbb{P}\bigl[\rho_{Z_0^{(1)}}(u) \geq r_1, \rho_{Z_0^{(2)}}(u) \geq r_2\bigr]
        &= \exp( - \gamma F_u(r_1,r_2) ),
    \end{align*}
    where
    \begin{align} \label{def:Fu}
        F_u(r_1,r_2)= F_u^{\Phi}(r_1,r_2)
        := \frac{1}{2} \int_{(S^{d-1})^2} |\langle r_1 u , v_1\rangle|+ |\langle r_2 u, v_2 \rangle| + |\langle r_1 u , v_1\rangle - \langle r_2 u, v_2 \rangle| \dint \Phi(v_1 , v_2) .
    \end{align}
\end{proposition}
\begin{remark}\label{rem:propFu}
         The proposition implies that for $i = 1, 2$, $\rho_{Z_0^{(i)}}(u)$ is exponentially distributed with rate $\gamma \int_{S^{d-1}} |\langle u , v\rangle|  \dint \phi_i (v) = 2 \gamma h(\Pi_i,u)$, where $\Pi_i$ is the normalized associated zonoid of $\Psi_i$. 
        This fact is well-known; see, for instance, the last paragraph of p.~89 of \cite{hugPoissonHyperplaneTessellations2024}.
\end{remark}
\begin{remark}\label{rem:invariance_of_jointDistribution}
    Let \(\Phi\) be an arbitrary coupling of \(\phi_1\) and \(\phi_2\) and let 
    \begin{align}\label{eq:def_phi_tilde}
        \widetilde{\Phi}:= \frac{1}{2}\bigl(\Phi + \Phi \circ T^{-1}\bigr),
    \end{align}
    be the symmetrized coupling associated to \(\Phi\), where \(\Phi \circ T^{-1}\) is the push-forward measure of \(\Phi\) under the antipodal map 
    \begin{align} \label{eq:def_antipodal}
        T:(v_1,v_2) \mapsto (-v_1,-v_2), \qquad (v_1,v_2) \in (S^{d-1})^2.
    \end{align}
    Observe that \(\widetilde{\Phi}\) is also a coupling of \(\phi_1\) and \(\phi_2\), and \(\widetilde{\Phi}\) is invariant under \(T\).
    Let \((Z_0^{(1)},Z_0^{(2)})\) and \((\widetilde{Z}_0^{(1)},\widetilde{Z}_0^{(2)})\) be the couplings defined by \eqref{eq:couplingZeroCells} using \(\Phi\) and \(\widetilde{\Phi}\) respectively.
    Since the integrand in \eqref{def:Fu} is invariant under \(T\), the integrals of \eqref{def:Fu} taken with respect to \(\Phi\) and \(\widetilde{\Phi}\) are equal, i.e., \(F_u^{\Phi}(r_1,r_2) = F_u^{\widetilde{\Phi}}(r_1,r_2)\).
    Therefore, for all \(r_1,r_2 \geq 0 \),
    \begin{align*}
       \mathbb{P}\bigl[\rho_{Z_0^{(1)}}(u) \geq r_1, \rho_{Z_0^{(2)}}(u) \geq r_2\bigr] 
       = 
       \mathbb{P}\bigl[\rho_{\widetilde{Z}_0^{(1)}}(u) \geq r_1, \rho_{\widetilde{Z}_0^{(2)}}(u) \geq r_2\bigr].
    \end{align*}
\end{remark}

We summarize some elementary properties of the function \(F_u\) in the following remark that will be used in the proofs throughout this section.
\begin{remark}\label{lem:PropertiesF}
 For any $u\in S^{d-1}$, $F_u$ satisfies the following properties:
        \begin{itemize}
            \item[(i)] increasing: $F_u(x_1,y_1) \leq F_u(x_2,y_2)$ for any $0\leq x_1 \leq x_2$ and $0\leq y_1 \leq y_2$;
            \item[(ii)] homogeneous: $F_u(\alpha x,\alpha y) = \alpha F_u(x,y)$, and in particular,
            $$F_u(x,0) = x F_u(1,0) = 2 x h (\Pi_1,u) \text{ and }F_u(0,y) = y F_u(0,1) = 2 y h (\Pi_2,u);$$
            \item[(iii)] convex with respect to each variable: 
            $x'\mapsto F_u(x',y)$ and $y'\mapsto F_u(x,y')$ are convex for all $x,y$;
            \item[(iv)] $F_u(x,y)\geq \frac{1}{2}(F_u(x,0) + F_u(0,y))$ for any nonnegative $x$, $y$;
        \end{itemize}
\end{remark}

\begin{proof}[Proof of Proposition \ref{prop:jointSurvival}]

For $u\in S^{d-1}$ and $r,r_1,r_2>0$ define
\begin{align*}
    A_{u,r}
    & := \{(v,t) \in S^{d-1} \times \mathbb{R} : H(v, t) \cap [0,r u] \neq \emptyset\} , \text{ and }
    \\
    A_{u,r_1,r_2}
    & := \{(v_1, v_2,t) \in (S^{d-1})^2 \times \mathbb{R} : (v_1,t) \in A_{u,r_1} \text{ or } (v_2,t) \in A_{u,r_2} \} ,
\end{align*}
so that we can write
\begin{align} \label{e:jointtail}
    \PP\left(\rho_{Z_0^{(1)}}(u) \geq r_1, \rho_{Z_0^{(2)}}(u) \geq r_2\right)
    &= \PP\left(\Xi\left(A_{u,r_1,r_2}\right) = 0 \right) 
    = \exp\left(-\EE\left[\Xi\left(A_{u,r_1,r_2}\right)\right] \right) .
\end{align}
By Campbell’s formula, the last expectation can be written in the following integral form
\begin{align*}
    \EE\left[\Xi \left(A_{u,r_1,r_2}\right)\right]
    &= \gamma \int_{(S^{d-1})^2} \int_{\R} \mathbf{1}_{(v_1, v_2,t) \in A_{u,r_1,r_2}} \dint t \dint \Phi(v_1 , v_2) .
\end{align*}
Note that, for any vectors $u,v\in S^{d-1}$ and any $t\in\R$, the hyperplane $H(v,t)$ intersects the line spanned by $u$ at a point $s u$ for some $s\in\R$ if and only if $t = s \langle u , v \rangle $.
It follows that for any $r > 0$,
\begin{align*}
    (v,t) \in A_{u,r}
    \Longleftrightarrow 
    \begin{cases}
        0\leq t \leq r \langle u , v \rangle & \text{if } \langle u , v \rangle \geq 0 \\
        0\geq t \geq r \langle u , v \rangle & \text{if } \langle u , v \rangle \leq 0. 
    \end{cases}
\end{align*}
Thus,
\begin{align*}
    (v_1,v_2,t) \in A_{u,r_1,r_2}
    \Longleftrightarrow 
    \begin{cases}
        0\leq t \leq \max ( \langle r_1 u , v_1 \rangle , \langle r_2 u , v_2 \rangle ) & \text{if } 0 \leq \langle u , v_1 \rangle , \langle u , v_2 \rangle \\
        \langle r_1 u , v_1 \rangle \leq t \leq \langle r_2 u , v_2 \rangle & \text{if } \langle u , v_1 \rangle \leq 0 \leq \langle u , v_2 \rangle\\
        \langle r_2 u , v_2 \rangle \leq t \leq \langle r_1 u , v_1 \rangle & \text{if } \langle u , v_2 \rangle \leq 0 \leq \langle u , v_1 \rangle \\
        \min ( \langle r_1 u , v_1 \rangle , \langle r_2 u , v_2 \rangle ) \leq t \leq 0 & \text{if } \langle u , v_1 \rangle , \langle u , v_2 \rangle \leq 0.
    \end{cases}
\end{align*}
Therefore, for any $v_1,v_2\in S^{d-1}$,    
\begin{align*}
    \int_{\mathbb{R}}\mathbf{1}_{A_{u,r_1,r_2}}(v_1,v_2,t) \dint t 
    & =
    \begin{cases}
        |\langle r_1 u, v_1 \rangle| +|\langle r_2u, v_2 \rangle|
        & \text{if } \langle u, v_1 \rangle \langle u, v_2 \rangle < 0
        \\ \max\{|\langle r_1u, v_1 \rangle|,|\langle r_2u, v_2 \rangle|\}
        & \text{if } \langle u, v_1 \rangle \langle u, v_2 \rangle \geq 0
    \end{cases}
    \\&= \frac{1}{2}\left(|\langle r_1 u, v_1 \rangle| +|\langle r_2u, v_2 \rangle| + |\langle r_1 u, v_1 \rangle - \langle r_2u, v_2 \rangle|\right). 
\end{align*}
It follows that
\begin{align*}
    \EE\left[\Xi \left(A_{u,r_1,r_2}\right)\right]
    &= \frac{\gamma}{2} \int_{(S^{d-1})^2} |\langle r_1 u, v_1 \rangle| +|\langle r_2u, v_2 \rangle| + |\langle r_1 u, v_1 \rangle - \langle r_2u, v_2 \rangle| \dint \Phi(v_1 , v_2),
\end{align*}
and substituting this into \Cref{e:jointtail} gives the desired result.
\end{proof}

\subsection{Tail distribution of the radial difference}

In this section, we consider the distribution of the positive part of $\rho_{Z_0^{(1)}}(u) - \rho_{Z_0^{(2)}}(u)$, for \(u \in S^{d-1}\). 
By symmetry (obtained by exchanging the indices \((1)\) and \((2)\)), the negative part can be treated the same way, and thus we will be able to derive moment bounds for $| \rho_{Z_0^{(1)}}(u) - \rho_{Z_0^{(2)}}(u) |$ in the sequel.

\begin{proposition}\label{prop:tailbnd}
   Let \(\Phi\) be a coupling of \(\phi_1\) and \(\phi_2\) which is antipodally symmetric, i.e., \(\Phi\) is invariant under the antipodal map \(T: (v_1,v_2) \mapsto (-v_1,-v_2)\) for all \((v_1,v_2) \in (S^{d-1})^2\).
    Given the coupling $(Z_0^{(1)},Z_0^{(2)})$ defined by \eqref{eq:couplingZeroCells} using \(\Phi\), we have for all $u \in S^{d-1}$ and $t > 0$,
    \begin{align*}
        \mathbb{P}[ \rho_{Z_0^{(1)}}(u)- \rho_{Z_0^{(2)}}(u) \geq t]
        &= 2 \gamma \int_{0}^{\infty} e^{-\gamma F_u(y+t,y)} \int_{(S^{d-1})^2} \langle u,v_2 \rangle   \indic{ 0 < \langle  u , v_2\rangle < \frac{y+t}{t} \langle u,v_2-v_1 \rangle} \dint \Phi (v_1 , v_2) \dint y,
    \end{align*}
    where $F_u(\cdot,\cdot)$ is defined as in \eqref{def:Fu}.
\end{proposition}

Before providing the full proof of the proposition, we state and prove two intermediate lemmas under the assumptions of \Cref{prop:tailbnd}.
In particular, in \Cref{lem:Fu_derivative}, we work under the antipodal-symmetry assumption on \(\Phi\).
The conclusion of \Cref{lem:conditionalSurvival} remains valid without this assumption.

 To make expressions lighter, we will use the notation \(X:=\rho_{Z_0^{(1)}}(u)\), \(Y :=\rho_{Z_0^{(2)}}(u) \).
The proof will proceed using the observation that
\begin{align} \label{eq:tail_expansion}
    \mathbb{P}[X-Y \geq t]
    &= \int_{0}^{\infty} \mathbb{P} \bigl[X \geq y+t \mid Y = y\bigr] \mathbb{P}_Y(\dint y)
\end{align}
where \(\mathbb{P}_Y\) is the distribution of \(Y\) and \(\mathbb{P}[X \geq x \mid Y=y]\) is the conditional survival function of \(X\) given \(Y=y\). The first lemma provides a closed form expression for this function.

\begin{lemma} \label{lem:conditionalSurvival}
    The conditional survival function of \(X\) given \(Y=y\) is given by 
    \begin{align*}
        \mathbb{P}\bigl[X \geq x \mid Y = y\bigr] 
        = g(x,y) :=
        \begin{cases}
            \frac{\partial_y F_u(x,y)}{F_u(0,1)} e^{- \gamma(F_u(x,y) - y F_u(0,1))}
            & \text{if $y'\mapsto F_u(x,y')$ is differentiable at $y$},
            \\0 & \text{otherwise}.
        \end{cases}
    \end{align*}
    In particular, this means that for any $x\geq 0$ and any Borel set $B\subset [0,\infty)$, we have
         \begin{align} \label{eq:ToProve}
              \mathbb{P}[X \geq x, Y \in B]
              & = \int_{B} g(x,y) \, \mathbb{P}_Y(\dint y).
         \end{align}
\end{lemma}
\begin{proof}
    Let $x\geq 0$.
    It is enough to establish \eqref{eq:ToProve} for sets $B=[y',\infty)$ with $y' > 0$.
    Since \(y' \mapsto F_u(x,y')\) is convex, the set of points where it is not differentiable is countable.
    Moreover \(\mathbb{P}_Y\) is absolutely continuous with respect to the Lebesgue measure (\Cref{rem:propFu} tells us it is exponentially distributed with rate $\gamma F_u(0,1)$) and therefore $y'\mapsto F_u(x,y')$ is differentiable $\mathbb{P}_Y$-almost-everywhere.
    Therefore,
    \begin{align*}
        \int_{y'}^{\infty}g(x,y) \, \mathbb{P}_Y(dy)  
        &= \int_{y'}^{\infty} e^{-\gamma(F_u(x,y) - yF_u(0,1))} 
            \frac{\partial_y F_u(x,y)}{F_u(0,1)} 
            \gamma F_u(0,1) e^{- \gamma F_u(0,1) y} \, \dint y \nonumber \\
        &= \int_{y'}^{\infty} e^{- \gamma F_u(x,y)} \gamma \partial_y F_u(x,y) \, \dint y \nonumber \\
        &=e^{-\gamma F_u(x,y')} = \mathbb{P}[X \geq x, Y \geq y'],
    \end{align*}
    where the last equality follows from \Cref{prop:jointSurvival}.
\end{proof}

\begin{lemma}\label{lem:Fu_derivative}
    Let $x\geq 0$ and $u\in S^{d-1}$.
    The map $y'\mapsto F_u(x,y')$ is differentiable at almost every $y'> 0$, in which case it satisfies
    \begin{align*}
        \partial_y\big|_{y=y'}F_u(x,y) 
        &=2  \int_{(S^{d-1})^2}  \langle u,v_2 \rangle  \indic{\langle x u , v_1\rangle < \langle y' u , v_2\rangle}  \indic{\langle u,v_2 \rangle > 0} \dint \Phi(v_1 , v_2).
    \end{align*}
\end{lemma}

\begin{proof}
    The convexity of $y'\mapsto F_u(x,y')$ justifies the almost everywhere differentiability.
    It also implies that \(\partial_{y}^+\big|_{y=y'}F_u(x,y)\) and \(\partial_{y}^-\big|_{y=y'}F_u(x,y)\) are always defined.
    We will first establish expressions for these one-sided derivatives.
    Recall the expression
    \begin{equation*}
        F_u(x,y) = \frac{1}{2} \int_{(S^{d-1})^2} \bigl( |\langle x u , v_1\rangle|+ |\langle y u, v_2 \rangle| + |\langle x u , v_1\rangle - \langle y u, v_2 \rangle| \bigr)  \dint \Phi(v_1 , v_2).
    \end{equation*}
    Observe that 
    \begin{align*}
        \partial_{y}^+\big|_{y=y'}F_u(x,y) &= \frac{F_u(0,1)}{2} + 
        \lim_{h \downarrow 0}\int_{(S^{d-1})^2}\frac{|\langle x u , v_1\rangle - \langle (y'+h) u, v_2 \rangle|- |\langle x u , v_1\rangle - \langle y' u, v_2 \rangle|}{2h} \dint \Phi(v_1,v_2).
    \end{align*}
    Note that the integrand on the right is uniformly bounded by \(\frac{|\langle u,v_2 \rangle|}{2}\) which is integrable.
    Thus we can apply the dominated convergence theorem.
    Observe also that, for any $v_1,v_2\in S^{d-1}$,
    \begin{align*}
        \lim_{h \downarrow 0} \frac{|\langle x u , v_1\rangle - \langle (y'+h) u, v_2 \rangle|- |\langle x u , v_1\rangle - \langle y' u, v_2 \rangle|}{h}
        & =
        \begin{cases}
            -\langle u , v_2 \rangle & \text{if } \langle x u , v_1\rangle > \langle y' u, v_2 \rangle, \\
            |\langle u , v_2 \rangle| & \text{if } \langle x u , v_1\rangle = \langle y' u, v_2 \rangle, \\
            \langle u , v_2 \rangle & \text{if } \langle x u , v_1\rangle < \langle y' u, v_2 \rangle . 
        \end{cases}
    \end{align*}
    Hence we get 
    \begin{align*}
        \partial_{y}^+\big|_{y=y'}F_u(x,y)
        = \frac{F_u(0,1)}{2} &+ \int_{\{ \langle x u , v_1\rangle \neq \langle y' u, v_2 \rangle \}}  \mathrm{sgn}(\langle y'u,v_2 \rangle - \langle xu,v_1 \rangle) \frac{\langle u,v_2 \rangle }{2} \dint \Phi(v_1,v_2) 
        \\&+ \int_{\{ \langle x u , v_1\rangle = \langle y' u, v_2 \rangle \}} \frac{|\langle u,v_2 \rangle|}{2} \dint \Phi(v_1,v_2).
    \end{align*}
    Similarly, we obtain
    \begin{align*}
        \partial_{y}^-\big|_{y=y'}F_u(x,y)
        = \frac{F_u(0,1)}{2} &+ \int_{\{ \langle x u , v_1\rangle \neq \langle y' u, v_2 \rangle \}}  \mathrm{sgn}(\langle y'u,v_2 \rangle - \langle xu,v_1 \rangle) \frac{\langle u,v_2 \rangle }{2} \dint \Phi(v_1,v_2) 
        \\&- \int_{\{ \langle x u , v_1\rangle = \langle y' u, v_2 \rangle \}} \frac{|\langle u,v_2 \rangle|}{2} \dint \Phi(v_1,v_2).
    \end{align*}
    Therefore, the partial derivative \(\partial_{y}\big|_{y=y'}F_u(x,y)\) at \(y'\) exists if and only if
    \begin{align}\label{eq:assumption_partial_derivative}
        \int_{\{ \langle x u , v_1\rangle = \langle y' u, v_2 \rangle \}} |\langle u, v_2 \rangle| \dint \Phi(v_1,v_2) = 0,
    \end{align}
    in which case it can be written
    \begin{align*}
        \partial_{y} \big|_{y=y'}F_u(x,y)
        &= \frac{F_u(0,1)}{2} + \int_{(S^{d-1})^2}  \mathrm{sgn}(\langle y'u,v_2 \rangle - \langle xu,v_1 \rangle) \frac{\langle u,v_2 \rangle }{2} \dint \Phi(v_1,v_2) 
        \\&= \frac{1}{2}\int_{(S^{d-1})^2} \left( \mathrm{sgn}(\langle y'u,v_2 \rangle - \langle xu,v_1 \rangle) \mathrm{sgn}(\langle u,v_2 \rangle) + 1 \right) | \langle u,v_2 \rangle | \, \dint \Phi(v_1,v_2) ,
    \end{align*}
    where the last equality follows from $F_u(0,1) =  \int_{(S^{d-1})^2} |\langle u, v_2 \rangle| \dint \Phi(v_1 , v_2) $.
    Since the integrand is zero when $\langle u,v_2 \rangle = 0$ and the integrand is invariant under \(T\), and since $\Phi$ is invariant under \(T\) by the hypothesis of \Cref{prop:tailbnd},
    multiplying the integrand by $2\, \mathbf{1}\{\langle u,v_2 \rangle >0\}$ does not change the value of the integral.
    That is,
    \begin{align*}
        \partial_{y} \big|_{y=y'}F_u(x,y)
        &= \int_{(S^{d-1})^2} \indic{\langle u,v_2 \rangle >0} \left( \mathrm{sgn}(\langle y'u,v_2 \rangle - \langle xu,v_1 \rangle) \mathrm{sgn}(\langle u,v_2 \rangle) + 1 \right)  | \langle u,v_2 \rangle | \, \dint \Phi(v_1,v_2) \\ 
        &= 2 \int_{(S^{d-1})^2} \indic{\langle u,v_2 \rangle >0} \indic{\langle x u, v_1 \rangle < \langle y' u , v_2\rangle} \langle u,v_2 \rangle \, \dint \Phi(v_1,v_2). 
    \end{align*}
\end{proof}

We are now in a position to prove \Cref{prop:tailbnd}.
\begin{proof}[Proof of \Cref{prop:tailbnd}]
    From Lemma \ref{lem:conditionalSurvival}, Lemma \ref{lem:Fu_derivative}, and the exponential distribution of $Y$, for any $x\geq 0$ and almost every $y\geq 0$, we get
    \begin{align*}
        &\mathbb{P}\bigl[X \geq x \mid Y = y\bigr] \frac{\mathbb{P}_Y(\dint y)}{\dint y}
        \\&= e^{-\gamma (F_u(x,y)-y F_u(0,1))}\frac{2  \int_{(S^{d-1})^2} \langle u,v_2 \rangle \indic{x\langle  u , v_1\rangle < y \langle  u , v_2\rangle}  \indic{\langle u,v_2 \rangle > 0} \dint \Phi (v_1 , v_2)}{F_u(0,1)} \gamma e^{-\gamma y F_u(0,1)} F_u(0,1) 
        \\&= 2 \gamma e^{-\gamma F_u(x,y)} \int_{(S^{d-1})^2} \langle u,v_2 \rangle  \indic{x\langle  u , v_1\rangle < y \langle  u , v_2\rangle}  \indic{\langle u,v_2 \rangle > 0} \dint \Phi (v_1 , v_2) . 
    \end{align*}    
    Therefore,
    \begin{align}
        \mathbb{P}(X-Y\geq t) \notag
        &= \int_0^\infty \mathbb{P}\left[X \geq y+t \mid Y=y\right] \mathbb{P}_Y (\dint y)
        \\&= 2 \gamma \int_0^\infty e^{-\gamma F_u(y+t,y)} \int_{(S^{d-1})^2} \langle u,v_2 \rangle  \indic{(y+t)\langle  u , v_1\rangle < y \langle  u , v_2\rangle}  \indic{\langle u,v_2 \rangle > 0} \dint \Phi (v_1 , v_2)  \dint y . \label{eq:08102025}
    \end{align}
    Finally, note that, by adding $t\langle  u , v_2\rangle - (y+t)\langle  u , v_1\rangle$ and dividing by $t$ both sides of the inequality, we have the equivalence
    \[ (y+t)\langle  u , v_1\rangle < y \langle  u , v_2\rangle
    \Leftrightarrow \langle  u , v_2\rangle < \frac{y+t}{t} \langle  u , v_2-v_1\rangle \]
    for any $v_1,v_2\in S^{d-1}$ and $y,t> 0$.
    Hence \Cref{prop:tailbnd} follows immediately from \eqref{eq:08102025}.
\end{proof}

\subsection{Moments of the radial difference}

Using the tail bound in the previous section and the formula $\mathbb{E} X = \int_0^\infty \mathbb{P}(X\geq t) \dint t$, we will now obtain moment bounds for the difference between the radial functions.

\begin{proposition}\label{thm:int_bnd_p}
    Let \(\Phi\) be an arbitrary coupling of \(\phi_1\) and \(\phi_2\), and let $p\geq 1$.
    Given the coupling $(Z_0^{(1)},Z_0^{(2)})$ defined by \eqref{eq:couplingZeroCells} using \(\Phi\),
    we have
    \begin{align}\label{eq:moments_radial}
        \int_{S^{d-1}} \mathbb{E}  (\rho_{Z_0^{(1)}}(u) - \rho_{Z_0^{(2)}}(u))^p_+ \sigma(\dint u)
        &\leq
        \frac{\EE_{\Phi}[\|V_2 - V_1\|]}{\gamma^p}  \int_{S^{d-1}} \frac{\Gamma(p+1)}{(h(\Pi_1,u) + h(\Pi_2,u)) h(\Pi_1,u)^p}  \sigma(\dint u). 
    \end{align}
\end{proposition}
\begin{proof}
    Let \(\tilde{\Phi}\) be the symmetrized coupling associated to \(\Phi\) defined by \eqref{eq:def_phi_tilde}.
    By \Cref{rem:invariance_of_jointDistribution}, for every \(u \in S^{d-1}\), the radial pairs induced by \(\Phi\) and \(\widetilde{\Phi}\) have the same joint distribution.
    Consequently, the expectations appearing on the left-hand side of \eqref{eq:moments_radial} taken with respect to \(\Phi\) and \(\widetilde{\Phi}\) are equal for every fixed \(u \in S^{d-1}\).
    Also, the expectation on the right-hand side of \eqref{eq:moments_radial} satisfies \(\mathbb{E}_{\Phi}[\|V_1-V_2\|] = \mathbb{E}_{\widetilde{\Phi}}[\|V_1-V_2\|]\).
    Therefore, without loss of generality, we assume that \(\Phi\) is invariant under the antipodal map \(T\) (defined by \eqref{eq:def_antipodal}) and use \Cref{prop:tailbnd} to obtain \eqref{eq:moments_radial}.

    First recall that for any non-negative random variable $X$, $$\mathbb{E} X^p = \int_0^\infty \mathbb{P}(X^p\geq s) \dint s = \int_0^\infty \mathbb{P}(X\geq t) p t^{p-1}\dint t.$$ By  \Cref{prop:tailbnd}, we then have for any $u\in S^{d-1}$, 
    \begin{align*}
        &\mathbb{E}  (\rho_{Z_0^{(1)}}(u) - \rho_{Z_0^{(2)}}(u))^p_+
        \\&= \int_0^\infty  2 \gamma \int_{0}^{\infty} e^{-\gamma F_u(y+t,y)} \int_{(S^{d-1})^2} \langle u,v_2 \rangle \indic{ 0 < \langle  u , v_2\rangle < \frac{y+t}{t} \langle u,v_2-v_1 \rangle} \dint \Phi (v_1 , v_2) \dint y \,  p t^{p-1} \dint t\\
        &= 2 p \gamma \int_1^\infty \left(\int_{0}^{\infty} t^p e^{-\gamma F_u(tz,t(z-1))} \dint t \right) \int_{(S^{d-1})^2} \langle u,v_2 \rangle \indic{ 0 < \langle  u , v_2\rangle < z \langle u,v_2-v_1 \rangle} \dint \Phi (v_1 , v_2) \dint z,
    \end{align*}
    where the second equality follows from the substitution $y=t(z-1)$ and Fubini's theorem.
    Using the homogeneity property of $F_u$ (\Cref{lem:PropertiesF}), and a substitution $s = \gamma F_u(z,z-1) t$, we find that the integral with respect to $t$ evaluates to 
    \begin{align*}
        \int_{0}^{\infty} t^p e^{-\gamma t F_u(z,z-1)} \dint t
        &= \frac{1}{(\gamma F_u(z,z-1))^{p+1}} \int_{0}^{\infty} s^p e^{-s} \dint s
        =\frac{\Gamma(p+1)}{(\gamma F_u(z,z-1))^{p+1}}, 
    \end{align*}
    and thus
    \begin{align*}
        &\mathbb{E} (\rho_{Z_0^{(1)}}(u) - \rho_{Z_0^{(2)}}(u))_+^p\\
        &= \frac{2 p \Gamma(p+1)}{\gamma^p} \int_1^\infty \frac{1}{F_u(z,z-1)^{p+1}} \left( \int_{(S^{d-1})^2} \langle u , v_2 \rangle \indic{0 < \langle u , v_2 \rangle < z \langle u , v_2 - v_1 \rangle } \dint \Phi(v_1,v_2) \right) \dint z.
    \end{align*}    
    By Fubini's theorem, we thus have
    \begin{align*}
        &\int_{S^{d-1}} \mathbb{E}  (\rho_{Z_0^{(1)}}(u) - \rho_{Z_0^{(2)}}(u))^p_+ \sigma(\dint u)\\
        &= \frac{2 p \Gamma(p+1)}{\gamma^p} \int_{S^{d-1}} \int_{(S^{d-1})^2}  \left(  \int_1^\infty \frac{1}{F_u(z,z-1)^{p+1}} \indic{\langle u , v_2 \rangle < z \langle u , v_2 - v_1 \rangle } \dint z \right) (\langle u , v_2 \rangle)_+ \dint \Phi(v_1,v_2) \sigma(\dint u). 
    \end{align*}
    Using the properties of $F_u$ summarized in \Cref{lem:PropertiesF}, we observe that
    \begin{align*} 
        F_u(z,z-1) 
        &\geq \frac{F_u(z,0)+F_u(0,z-1)}{2}
        = \frac{z}{2}(F_u(1,0) + F_u(0,1)) - \frac{1}{2}F_u(0,1) ,
    \end{align*}
    from which we get
    \begin{align*}
       &\int_1^\infty \frac{1}{F_u(z,z-1)^{p+1}} \indic{\langle u , v_2 \rangle < z \langle u , v_2 - v_1 \rangle } \dint z \\
       &\leq  \int_1^\infty \frac{2^{p+1}}{\left[z(F_u(1,0) + F_u(0,1)) - F_u(0,1)\right]^{p+1}}  \indic{\langle u , v_2 \rangle < z \langle u , v_2 - v_1 \rangle }  \dint z 
       \\&= \frac{2^{p+1}}{F_u(1,0) + F_u(0,1)}\int_{F_u(1,0)}^\infty \omega^{-p-1}  \indic{ \langle u , v_2 \rangle < \frac{\left(\omega + F_u(0,1)\right)}{F_u(1,0) + F_u(0,1)} \langle u , v_2 - v_1 \rangle} \dint \omega , 
    \end{align*}
    where the equality follows from the change of variable $\omega = z(F_u(1,0) + F_u(0,1)) - F_u(0,1)$.
    Evaluating the last integral we see that
    \begin{align*}
        &\int_{F_u(1,0)}^\infty \omega^{-p-1}  \indic{ \langle u , v_2 \rangle < \frac{\left(\omega + F_u(0,1)\right)}{F_u(1,0) + F_u(0,1)} \langle u , v_2 - v_1 \rangle }  \dint \omega \\
        &= \frac{1}{p}\left( \max \left\{F_u(1,0),   \frac{\langle u , v_2 \rangle}{\langle u , v_2 - v_1 \rangle}(F_u(1,0) + F_u(0,1)) - F_u(0,1) \right\}\right)^{-p}  \\
        &\leq \frac{1}{p} \left( \max \left\{ 1 , \frac{\langle u , v_2 \rangle}{\langle u , v_2 - v_1 \rangle} \right\} F_u(1,0) \right)^{-p} \\
        &= \frac{1}{p F_u(1,0)^p} \min \left\{ 1 , \left(\frac{\langle u , v_2 - v_1 \rangle}{\langle u , v_2 \rangle}\right)^p_+ \right\}\\
        &\leq \frac{1}{p F_u(1,0)^p} \min \left\{ 1 , \left(\frac{\langle u , v_2 - v_1 \rangle}{\langle u , v_2 \rangle}\right)_+ \right\}\\
        &\leq \frac{1}{p F_u(1,0)^p} \frac{\| v_2 - v_1 \|}{(\langle u , v_2 \rangle)_+}. 
    \end{align*}
    The first inequality above follows in the case $\frac{\langle u , v_2 \rangle}{\langle u , v_2 - v_1 \rangle} \geq 1$ by the fact that $x \mapsto x^{-p}$ is decreasing. For the case $\frac{\langle u , v_2 \rangle}{\langle u , v_2 - v_1 \rangle} < 1$ we also have
    \begin{align*}
        \max \left\{F_u(1,0),   \frac{\langle u , v_2 \rangle}{\langle u , v_2 - v_1 \rangle}(F_u(1,0) + F_u(0,1)) - F_u(0,1) \right\} & = F_u(1,0) =  F_u(1,0) \max \left\{ 1 , \frac{\langle u , v_2 \rangle}{\langle u , v_2 - v_1 \rangle} \right\}.
    \end{align*}
    
    Thus, for any $u\in S^{d-1}$,
    \begin{align*}
        &\int_{(S^{d-1})^2}  \left(  \int_1^\infty \frac{1}{F_u(z,z-1)^{p+1}} \indic{\langle u , v_2 \rangle < z \langle u , v_2 - v_1 \rangle } \dint z \right) (\langle u , v_2 \rangle)_+ \dint \Phi(v_1,v_2) 
        \\& \leq \frac{2^{p+1}}{p(F_u(1,0) + F_u(0,1)) F_u(1,0)^p}  \int_{(S^{d-1})^2} \|v_1-v_2\| \indic{\langle u,v_2 \rangle >0} \dint \Phi(v_1,v_2) 
        \\& = \frac{2^p}{p(F_u(1,0) + F_u(0,1)) F_u(1,0)^p}  \int_{(S^{d-1})^2} \|v_1-v_2\| \dint \Phi(v_1,v_2),         
    \end{align*}    
    where the equality follows from the fact that \(\Phi \circ T^{-1} = \Phi\) and $\|(-v_1)-(-v_2)\|= \|v_1-v_2\|$.
    Hence
    \begin{align*}
        &\int_{S^{d-1}} \mathbb{E}  (\rho_1(u) - \rho_2(u))^p_+ \sigma(\dint u)
        \\&\leq \frac{2 p \Gamma(p+1)}{\gamma^p} \int_{S^{d-1}} \frac{2^p}{p(F_u(1,0) + F_u(0,1)) F_u(1,0)^p}  \int_{(S^{d-1})^2} \|v_1-v_2\| \dint \Phi(v_1,v_2)  \sigma(\dint u) 
        \\&= \frac{\EE_{\Phi}[\|V_2 - V_1\|]}{\gamma^p}  \int_{S^{d-1}} \frac{2^{p+1} \Gamma(p+1)}{(F_u(1,0) + F_u(0,1)) F_u(1,0)^p}  \sigma(\dint u). 
    \end{align*}
    Recalling that $F_u(1,0)=2h(\Pi_1,u)$ and  $F_u(0,1)=2h(\Pi_2,u)$ yields the proof.
\end{proof}

\section{Upper bound for Wasserstein distance between zero cells}\label{sec:main_results}
Recall the notation of \Cref{sec:background}.
In particular, recall that \(W_{\rho_p}\) and \(W_{\Delta}\), defined by \eqref{eq:Wasserstein_p}, \eqref{eq:radial_metric} and \eqref{def:symmDiff}, denote the Wasserstein distances with respect to the radial \(p\)-metric and the symmetric-difference volume metric, respectively. Before stating our main result, \Cref{thm:Wass_p_bnd_new}, we obtain the following core lemma giving an upper bound on the Wasserstein distance between zero cells depending on the Wasserstein distance \(W_{\text{geo}}\) between the directional distributions.

\begin{lemma} \label{lem:mainBound}
    Fix \(\gamma>0\), let \(\phi_1,\phi_2\) be nondegenerate even probability measures on $S^{d-1}$, and define stationary Poisson hyperplane processes \(\Psi_1,\Psi_2\) with common intensity \(\gamma\) and directional distributions \(\phi_1,\phi_2\), respectively.
    Let $Z_0(\Psi_1)$ and $Z_0(\Psi_2)$ be the zero cells of \(\Psi_1\) and \(\Psi_2\).
    Then, for $p \in [1, \infty)$,
    \begin{align*}
        W_{\rho_p}\bigl(Z_0(\Psi_1),Z_0(\Psi_2) \bigr) 
        \leq \frac{K_p(\phi_1,\phi_2)}{\gamma} W_{\mathrm{geo}}(\phi_1,\phi_2)^{\frac{1}{p}}, 
    \end{align*} 
    and    
    \begin{align*}
        W_{\Delta}(Z_0(\Psi_1),Z_0(\Psi_2)) 
        \leq \frac{K_\Delta(\phi_1,\phi_2)}{\gamma^d} W_{\mathrm{geo}}(\phi_1, \phi_2)^{\frac{1}{d}},
    \end{align*}
    where

\begin{align} 
        \label{def:Kp}
        K_p(\phi_1,\phi_2)
        &:= \left(\Gamma(p+1)
        \int_{S^{d-1}} \frac{h(\Pi_1,u)^p + h(\Pi_2,u)^p}{(h(\Pi_1,u) + h(\Pi_2,u)) h(\Pi_1,u)^p h(\Pi_2,u)^p} \sigma(\dint u) \right)^{\frac{1}{p}}
    \end{align}
    and
    \begin{align}
        \label{def:KDelta}
        K_{\Delta} (\phi_1,\phi_2)
        &:= \frac{K_d (\phi_1,\phi_2)}{2^{d-1}}\left(d!\int_{S^{d-1}}
             h(\Pi_1,u)^{-d} + h(\Pi_2,u)^{-d} \sigma(\mathrm{d}u) \right)^{\frac{d-1}{d}},
    \end{align}
    where $K_d (\phi_1,\phi_2)$ is the constant \eqref{def:Kp} for $p = d$.
\end{lemma}

\begin{remark}
    If $p=1$ the first constant simplifies to 
    $K_1 
    = \int_{S^{d-1}} \frac{1}{h(\Pi_1,u) h(\Pi_2,u)} \sigma(\dint u).$
\end{remark}
\begin{proof}
 Let $\Phi$ be an optimal coupling of $\phi_1$ and $\phi_2$ so that 
\[W_{\mathrm{geo}}(\phi_1,\phi_2) 
= \mathbb{E}_{\Phi}[\mathrm{geo}(V_1,V_2)] 
\geq \mathbb{E}_{\Phi}[\|V_1-V_2\|].
\] 
Note that the latter inequality is satisfied by any coupling.

    Let $(Z_0^{(1)},Z_0^{(2)})$ be the coupling associated to \(\Phi\) as defined in \eqref{eq:couplingZeroCells}.
    By definitions \eqref{eq:Wasserstein_p}, \eqref{eq:radial_metric} and \eqref{def:symmDiff} of the Wasserstein distances $W_{\rho_p}$ and $W_{\Delta}$, we only need to show that
    \begin{align} \label{goal-p}
        \mathbb{E}   \left( \int_{S^{d-1}} | \rho_1(u) - \rho_2(u) |^p \sigma(\dint u) \right)^\frac{1}{p}
        &\leq \frac{K_p(\phi_1,\phi_2)}{\gamma} \EE_{\Phi}[\|V_2 - V_1\|]^{\frac{1}{p}},
    \end{align}
    and
    \begin{align} \label{goal-Delta}
        \frac{1}{d} \int_{S^{d-1}} \mathbb{E} | \rho_1(u)^d - \rho_2(u)^d | \sigma(\dint u) 
        &\leq \frac{K_\Delta(\phi_1,\phi_2)}{\gamma^d} \EE_{\Phi}[\|V_2 - V_1\|]^{\frac{1}{d}},
    \end{align}
    where $\rho_i(u) \coloneq \rho_{Z_0^{(i)}}(u)$.
    We start by using Jensen's inequality to bound the left hand side of \eqref{goal-p}.
    \begin{align}
         \label{eq:WpboundPosNeg}
        \mathbb{E}   \left( \int_{S^{d-1}} | \rho_1(u) - \rho_2(u) |^p \sigma(\dint u) \right)^\frac{1}{p}
        &\leq \left( \int_{S^{d-1}} \mathbb{E}  | \rho_1(u) - \rho_2(u) |^p \sigma(\dint u) \right)^\frac{1}{p} .
    \end{align}
    By \Cref{thm:int_bnd_p},
\begin{align*}
        \int_{S^{d-1}} \mathbb{E}  (\rho_1(u) - \rho_2(u))^p_+ \sigma(\dint u)
        &\leq \frac{\EE_{\Phi}[\|V_2 - V_1\|]}{\gamma^p}  
        \int_{S^{d-1}} \frac{ \Gamma(p+1)}{(h(\Pi_1,u) + h(\Pi_2,u)) h(\Pi_1,u)^p}  \sigma(\dint u).
    \end{align*}
    By exchanging the roles of indices \(1\) and \(2\), a similar bound is obtained for $\int_{S^{d-1}} \mathbb{E}  (\rho_1(u) - \rho_2(u))^p_- \sigma(\dint u)$, where we only need to swap $F_u(1,0)$ and $F_u(0,1)$, and it follows that

\begin{align}
        \label{common-bound-p-Delta}
        \left( \int_{S^{d-1}} \mathbb{E}  | \rho_1(u) - \rho_2(u) |^p \sigma(\dint u) \right)^\frac{1}{p}
        \leq \frac{K_p(\phi_1,\phi_2)}{\gamma} \EE_{\Phi}[\|V_2 - V_1\|]^{\frac{1}{p}}.
    \end{align}

    Combining this bound with \eqref{eq:WpboundPosNeg} gives \eqref{goal-p}.

    It remains to show \eqref{goal-Delta}.
    Observe that 
    \begin{align*}
        \big|\rho_1(u)^d - \rho_2(u)^d\big| 
        &= \big|\rho_1(u)-\rho_2(u)\big| \biggl(\sum_{i=0}^{d-1} \rho_1(u)^{d-1-i} \rho_2(u)^{i} \biggr)  \leq \big|\rho_1(u) - \rho_2(u)\big| d (\rho_1(u) \vee \rho_2(u))^{d-1},
    \end{align*}
    where \(\rho_1(u) \vee \rho_2(u)\) denotes \(\max\{\rho_1(u),\rho_2(u)\}\).
    Using this upper bound yields:
    \begin{align*}
        \frac{1}{d} \int_{S^{d-1}} \mathbb{E} | \rho_1(u)^d - \rho_2(u)^d | \sigma(\dint u)
        & \leq \int_{S^{d-1}} \mathbb{E}\bigl[\big|\rho_1(u) - \rho_2(u)\big| (\rho_{1}(u) \vee \rho_2(u))^{d-1} \bigr] \sigma(\mathrm{d}u).
    \end{align*}

    Applying Hölder's inequality to the expectation in the above expression gives:
    \begin{align}
        \notag
        &\frac{1}{d} \int_{S^{d-1}} \mathbb{E} | \rho_1(u)^d - \rho_2(u)^d | \sigma(\dint u)
        \\
        &\leq \int_{S^{d-1}} 
            \biggl(\mathbb{E}\bigl[\big|\rho_1(u) - \rho_2(u)\big|^d\bigr]\biggr)^{\frac{1}{d}} 
            \biggl(\mathbb{E}\bigl[\bigl(\rho_1(u) \vee \rho_2(u)\bigr)^d\bigr]\biggr)^{\frac{d-1}{d}}
            \sigma(\mathrm{d}u) \nonumber 
        \\
        &\leq 
            \biggl( \int_{S^{d-1}}
             \mathbb{E}\bigl[\big|\rho_1(u) - \rho_2(u)\big|^d\bigr] \sigma(\mathrm{d}u) 
            \biggr)^{\frac{1}{d}}
            \biggl( \int_{S^{d-1}}
             \mathbb{E}\bigl[\bigl(\rho_1(u) \vee \rho_2(u)\bigr)^d\bigr] \sigma(\mathrm{d}u)
            \biggr)^{\frac{d-1}{d}}, \label{eq:wass_symm_1}
    \end{align}
    where the last inequality follows from again applying Hölder's inequality 
    to the outer integral.
    Using the bound \((\rho_1(u) \vee \rho_2(u))^d \leq \rho_1(u)^d + \rho_2(u)^d\) and interchanging the integral and expectation,
    \begin{align*}
        \int_{S^{d-1}}
             \mathbb{E}\bigl[\bigl(\rho_1(u) \vee \rho_2(u)\bigr)^d\bigr] \sigma(\mathrm{d}u)
        &\leq 
        \int_{S^{d-1}}
             \mathbb{E}\bigl[\rho_1(u)^d \bigr] \sigma(\mathrm{d}u) 
        + \int_{S^{d-1}} \mathbb{E}\bigl[\rho_2(u)^d\bigr] \sigma(\mathrm{d}u) \nonumber \\
        &=\frac{d!}{(2\gamma)^d}
        \int_{S^{d-1}}
             h(\Pi_1,u)^{-d} + h(\Pi_2,u)^{-d} \sigma(\mathrm{d}u),
    \end{align*}
    where the equality follows from the fact that $\rho_i(u)$ is exponentially distributed with rate parameter $2 \gamma h(\Pi_i,u)$, see \Cref{rem:propFu}.
    Plugging this and \eqref{common-bound-p-Delta} (with $p=d$) into \eqref{eq:wass_symm_1} gives
    \begin{align*}
        &\frac{1}{d} \int_{S^{d-1}} \mathbb{E} | \rho_1(u)^d - \rho_2(u)^d | \sigma(\dint u) \\
        &\leq \frac{K_d(\phi_1,\phi_2)}{\gamma^d} \left(\frac{d!}{2^d}\int_{S^{d-1}}
             \left[h(\Pi_1,u)^{-d} + h(\Pi_2,u)^{-d}\right] \sigma(\mathrm{d}u) \right)^{\frac{d-1}{d}} \EE_{\Phi}[\|V_2 - V_1\|]^{\frac{1}{d}} ,
    \end{align*}
    which is precisely the result \eqref{goal-Delta}.
\end{proof}

To ensure local stability of the distribution of the zero cell, the upper bound on the Wasserstein distance between the zero cells should become small as the distance between directional distributions becomes small. We show this is indeed the case in our main result below.

\begin{theorem}\label{thm:Wass_p_bnd_new}  Consider the setting of \Cref{lem:mainBound}.  
Assume that $\phi_2$ is sufficiently close to $\phi_1$ such that
    \begin{align}\label{e:closetophi1-bnd} 
    W_{\mathrm{geo}}(\phi_1,\phi_2) \leq \min_{u\in S^{d-1}} h(\Pi_1,u) \eqcolon C_1 .
    \end{align}
Then, for $p \in [1, \infty)$,
\begin{align}\label{e:Wrhop-holderbnd}
         W_{\rho_p}\bigl(Z_0^{(1)} , Z_0^{(2)} \bigr) 
        \leq \frac{ 6\Gamma(p+1)^{1/p}(d\kappa_d)^{1/p} C_1^{-\frac{p+1}{p}} }{\gamma} W_{\mathrm{geo}}(\phi_1,\phi_2)^{\frac{1}{p}}.   
\end{align}
and
\begin{align}\label{e:Wdelta-holderbnd} 
    W_{\Delta}\bigl(Z_0^{(1)},Z_0^{(2)}\bigr) 
    &\leq 
    \frac{ 6d\kappa_d\,d!\, \left(1+2^{-d}\right)^{\frac{d-1}{d}} C_1^{-d-\frac1d} }{\gamma^d} W_{\mathrm{geo}}(\phi_1,\phi_2)^{\frac1d}, 
\end{align}
where \(\kappa_d:= \lambda_d(B_d)\) is the volume of the unit ball in \(\mathbb{R}^d\).
\end{theorem}

\begin{remark}
    Note that in \Cref{thm:Wass_p_bnd_new}, the dependence in the upper bound on the distance $W_{\mathrm{geo}}(\phi_1, \phi_2)$ weakens as $p$ increases, and the stated local H\"{o}lder continuity estimate degenerates as $p \to \infty$. Thus, local H\"{o}lder continuity for the radial metric for $p = \infty$ remains an open problem.
\end{remark}

\begin{proof}
    We start by showing that for any unit vector $u$, $h(\Pi_2,u)$ can be approximated by $h(\Pi_1,u)$ up to an error factor in the range $[1/2,3/2]$.
    To obtain this, we first note that, for any coupling $\Phi$ of $\phi_1$ and $\phi_2$, we have
    \begin{align*}
        |h(\Pi_2,u) - h(\Pi_1,u)|
        &= \frac{1}{2}  \left| \int_{S^{d-1}} |\langle u , v_1 \rangle| \phi_1(\dint v_1) - \int_{S^{d-1}} |\langle u , v_2 \rangle| \phi_2(\dint v_2) \right| \\
        &= \frac{1}{2}  \left| \int_{S^{d-1}\times S^{d-1}} |\langle u , v_1 \rangle| - |\langle u , v_2 \rangle| \Phi(\dint v_1 , \dint v_2) \right| \\
        &\leq \frac{1}{2}  \int_{S^{d-1}\times S^{d-1}} |\langle u , v_1 \rangle - \langle u , v_2 \rangle|  \Phi(\dint v_1 , \dint v_2)
        \\&\leq \frac{1}{2} \EE_\Phi \|V_2-V_1\|.
    \end{align*}    
    Minimizing over all couplings and using our assumption \eqref{e:closetophi1-bnd}, we get
    \begin{align*}
        |h(\Pi_2,u) - h(\Pi_1,u)|
        &\leq \frac{1}{2} W_{\mathrm{geo}}(\phi_1,\phi_2)
        \leq \frac{h(\Pi_1,u)}{2} ,
    \end{align*}
    from which it follows that
    \begin{align}
        \label{eq:boundsOnh}
        \frac{1}{2} h(\Pi_1,u) 
        \leq h(\Pi_2,u)
        \leq \frac{3}{2} h(\Pi_1,u) .
    \end{align}
    From these bounds we have
    \begin{align}
        \label{boundOnKp}
        \frac{h(\Pi_1,u)^p + h(\Pi_2,u)^p}{(h(\Pi_1,u) + h(\Pi_2,u)) h(\Pi_1,u)^p h(\Pi_2,u)^p} 
        &\leq \frac{ (1+(3/2)^p) h(\Pi_1,u)^p}{ (3/2) (1/2^p) h(\Pi_1,u) h(\Pi_1,u)^p h(\Pi_1,u)^p} 
        \leq \frac{6^p}{h(\Pi_1,u)^{p+1}} .
    \end{align}
    Plugging this into \eqref{def:Kp}, we obtain
    \begin{align*} 
        K_p(\phi_1,\phi_2) &\leq 6\Gamma(p+1)^{1/p} C_1^{-\frac{p+1}{p}} (d\kappa_d)^{1/p}, 
    \end{align*}
    which, by \Cref{lem:mainBound}, gives \eqref{e:Wrhop-holderbnd}.
    From \eqref{eq:boundsOnh} we also get
    \begin{align*} 
       \int_{S^{d-1}} \frac{1}{h(\Pi_1,u)^d} + \frac{1}{h(\Pi_2,u)^d} \sigma(\dint u)
        & \leq \int_{S^{d-1}} \frac{1+2^d}{h(\Pi_1,u)^d} \sigma(\dint u).
    \end{align*}
    Therefore, with the definition \eqref{def:KDelta} of $K_\Delta$ and the bound \eqref{boundOnKp}, we obtain

    \begin{align*}
        K_{\Delta}
        &\leq
        \frac{
        6(d!)^{1/d}(d\kappa_d)^{1/d}
        C_1^{-\frac{d+1}{d}}
        }{2^{d-1}}
        \left(
        (d\kappa_d)d!
        \frac{1+2^d}{C_1^d}
        \right)^{\frac{d-1}{d}}
        =
        6d\kappa_d\,d!\,
        \left(1+2^{-d}\right)^{\frac{d-1}{d}}
        C_1^{-d-\frac1d}.
    \end{align*}

    By \Cref{lem:mainBound} we thus obtain \eqref{e:Wdelta-holderbnd} which completes the proof.
\end{proof}

\section{Application to density estimation}\label{sec:densityestimation}

A natural approach for nonparametric regression and density estimation is to partition the set of observations and compute statistics locally within each cell. These histogram estimates are classical methods in nonparametric statistics \cite[Chapter 4]{gyorfi2002distribution}. The simplest version of a histogram estimator partitions the space into hypercubes with width based on the number of available training data points, but not on the distribution of the data itself. Existing theoretical guarantees for these methods include sufficient conditions for $L^1$ consistency of the density estimator \cite{devroye1985nonparametric} and for nonparametric regressors and classifiers \cite{devroye1983distribution}.

Many variants of the histogram approach introduce randomization into the construction of the partition of the input space of the regression function or the support of the probability density. Examples include random binning features \cite{Rahimi} and the Mondrian process \cite{roy2008mondrian, lakshminarayanan2014mondrian} as well as partitions that allow for oblique split directions such as random projection trees \cite{dasgupta2008random} and partitions generated by stable under iteration (STIT) processes \cite{OReillyTran2021, OReillyTran2021minimax}. Randomized partitioning estimates can achieve state-of-the-art empirical performance in many applications when combined with data-adaptive mechanisms for choosing features along which to make splits as in the CART and C4.5 trees that generate random forests \cite{breiman2001random}. 

In this section, we see how the results in this paper can be used to prove a stability result for a class of random partition estimators with respect to perturbations in the distribution of the random partition in the setting of density estimation. The random partition will be the random tessellation generated from a stationary Poisson hyperplane process or hierarchical STIT process \cite{Nagel2005}, and the focus will be on perturbations of the directional distribution determining the linear combination of covariates used to make splits. We recall that for an STIT process with lifetime $\gamma$ and directional distribution $\phi$, the zero cell has the same distribution as the zero cell of a Poisson hyperplane process with intensity $\gamma$ and directional distribution $\phi$ \cite{Thale2013Poisson}.

Let $\mathcal{D}_n := \{X_1, \ldots, X_n\}$ be $n$ i.i.d. samples drawn from a distribution with unknown density $f$ in $\RR^d$. We can define a density estimator for $f$ from a random tessellation $\mathcal{P}$ of $\RR^d$ as 
\begin{align}\label{e:tree_estimate}
\hat{f}_{n}(x) := \frac{1}{n} \sum_{i=1}^n \frac{1_{\{x \in Z_{X_i}\}}}{\mathrm{vol}(Z_{X_i})} = \frac{1}{n} \sum_{i=1}^n \frac{1_{\{X_i \in Z_x\}}}{\mathrm{vol}(Z_{x})}, \quad x \in \RR^d,
\end{align}
where $\mathrm{vol}(Z)$ denotes the $d$-dimensional volume of a cell $Z \subset \RR^d$, and $Z_x$ is the cell of the tessellation that contains $x$.
Dividing by the volume in each term of the sum ensures that \(\hat{f}_n(x)\) integrates to one, that is, it is indeed a density. 

We are interested in the stability of this estimate with respect to perturbations in the distribution of the tessellation. 
In particular, the main results of this section provide upper bounds on the expected \(L^1\) distance between density estimators constructed from two stationary Poisson hyperplane tessellations in terms of the Wasserstein distance between their directional distributions.

\begin{proposition}\label{prop:density-robust}
    Let $\hat{f}^{(1)}_{\gamma, n}$ and $\hat{f}^{(2)}_{\gamma, n}$ be the estimators of a density $f$ in $\mathbb{R}^d$ for $d \geq 2$ as in \eqref{e:tree_estimate} obtained from two stationary Poisson hyperplane tessellations or STIT tessellations $\mathcal{P}_1(\gamma)$ and $\mathcal{P}_2(\gamma)$ with intensity/lifetime parameter $\gamma$ and directional distributions $\phi_1$ and $\phi_2$, respectively. Assume the tessellations are independent of the data samples $\mathcal{D}_n$.
    Then, there exist random probability densities \(\hat{g}_{\gamma,n}^{(1)}\) and \(\hat{g}_{\gamma,n}^{(2)}\), defined on a common probability space, such that, for every fixed \(x \in \mathbb{R}^d\),
    \[ 
    \hat{g}_{\gamma,n}^{(i)}(x) 
    \overset{(d)}{=} 
    \hat{f}_{\gamma,n}^{(i)}(x), \qquad i=1,2, 
    \]
     and such that, for all $\phi_2$ satisfying assumption \eqref{e:closetophi1-bnd}, we have
    \begin{align*}
        2 \mathbb{E}[d_{\operatorname{TV}}(\hat{g}^{(1)}_{\gamma, n},\hat{g}^{(2)}_{\gamma, n})]
        =
    \EE \left[\int_{\mathbb{R}^d} \left|\hat{g}^{(1)}_{\gamma, n}(x) - \hat{g}^{(2)}_{\gamma, n}(x)\right| \dint x\right]
    &\leq C(\phi_1, d) W_{\mathrm{geo}}(\phi_1,\phi_2)^{1/d}, 
\end{align*}
for some constant $C(\phi_1,d)$ depending only on $\phi_1$ and $d$, where the expectation is taken with respect to the random tessellations and $\mathcal{D}_n$.
\end{proposition}

\begin{proof}
 Let \(\Phi\) be the optimal coupling of \(\phi_1\) and \(\phi_2\) that attains the Wasserstein metric: 
 \[W_{\mathrm{geo}}(\phi_1,\phi_2) 
= \mathbb{E}_{\Phi}[\mathrm{geo}(V_1,V_2)] 
\geq \mathbb{E}_{\Phi}[\|V_1-V_2\|],
\]
where \((V_1,V_2) \sim \Phi\).
Let $(Z_0^{(1)}, Z_0^{(2)})$ be the coupling of zero cells from the coupling \eqref{eq:couplingZeroCells} of stationary Poisson hyperplane tessellations $\mathcal{P}_1(1)$ and $\mathcal{P}_2(1)$ with unit intensity and with the coupling \(\Phi\) of directional distributions $\phi_1$ and $\phi_2$, respectively. For each $i = 1,2$, define 

\begin{align}\label{e:hatgn}
    \hat{g}^{(i)}_{\gamma, n}(x) := \frac{1}{n} \sum_{j=1}^n \frac{1_{\{X_j \in \gamma^{-1}Z^{(i)}_0 + x\}}}{\mathrm{vol}(\gamma^{-1}Z_0^{(i)} + x)
    },\; \forall x \in \mathbb{R}^d.
\end{align}
By stationarity and the scaling property \cite[Lemma 5]{Nagel2005}, $\frac{1}{\gamma}Z^{(i)}_0 + x \overset{{(d)}}{=} Z_{x, \gamma}^{(i)}$ for all \(x \in \mathbb{R}^d\), where $Z_{x, \gamma}^{(i)}$ is the cell of  $\mathcal{P}_i(\gamma)$ containing $x$.
Consequently, for every fixed $x\in\mathbb{R}^d$, 
\[
 \hat{g}_{\gamma,n}^{(i)}(x) 
 \overset{(d)}{=} 
 \hat{f}_{\gamma,n}^{(i)}(x), \qquad i=1,2. 
 \]

To show the upper bound in the proposition, we first see that the triangle inequality and a change of variables $y = \frac{z}{\gamma} + x$ give 
\begin{align}
   \EE \left[\int_{\mathbb{R}^d} \left|\hat{g}^{(1)}_{\gamma, n}(x) - \hat{g}^{(2)}_{\gamma, n}(x)\right| \dint x\right] 
   & \leq  \EE \left[\int_{\mathbb{R}^d} \int_{\RR^d} f(y) \left|\frac{1_{\{y \in \gamma^{-1}Z^{(1)}_0 + x\}}}{\mathrm{vol}(\gamma^{-1}Z^{(1)}_{0} + x)} - \frac{1_{\{y \in \gamma^{-1}Z^{(2)}_0 + x\}}}{\mathrm{vol}(\gamma^{-1}Z^{(2)}_{0}+x)}\right| \dint y \dint x\right] \nonumber \\
   &= \EE \left[ \int_{\RR^d} \int_{\mathbb{R}^d}  f\left(\frac{z}{\gamma} + x\right) \left|\frac{1_{\{z \in Z^{(1)}_0\}}}{\mathrm{vol}(Z^{(1)}_{0})} - \frac{1_{\{z \in Z^{(2)}_0\}}}{\mathrm{vol}(Z^{(2)}_{0})}\right| \dint x \dint z \right] \nonumber \\
   &= \EE \left[\int_{\RR^d}\left|\frac{1_{\{z \in Z^{(1)}_0\}}}{\mathrm{vol}(Z^{(1)}_{0})} - \frac{1_{\{z \in Z^{(2)}_0\}}}{\mathrm{vol}(Z^{(2)}_{0})}\right| \dint z \right],
\end{align}
where the last equality follows from taking the integral $\int_{\mathbb{R}^d} f(\frac{z}{\gamma} + x) \dint x  = 1$.
Next, we see that
\begin{align}\label{e:voldiff_frac}
    &\left|\frac{1_{\{z \in Z^{(1)}_0\}}}{\mathrm{vol}(Z^{(1)}_{0})} - \frac{1_{\{z \in Z^{(2)}_0\}}}{\mathrm{vol}(Z^{(2)}_{0})}\right|  = \frac{\left|\mathrm{vol}(Z^{(2)}_{0})1_{\{z \in Z^{(1)}_0\}} - \mathrm{vol}(Z^{(1)}_{0})1_{\{z \in Z^{(2)}_{0}\}}\right|}{\mathrm{vol}(Z^{(1)}_{0})\mathrm{vol}(Z^{(2)}_{0})} \nonumber \\
    &= \frac{\left|\mathrm{vol}(Z^{(2)}_{0}) - \mathrm{vol}(Z^{(1)}_{0})\right|}{\mathrm{vol}(Z^{(1)}_{0})\mathrm{vol}(Z^{(2)}_{0})}1_{\{z \in Z_0^{(1)} \cap Z_0^{(2)}\}} + \frac{1_{\{z \in Z_0^{(1)} \cap (Z_0^{(2)})^c\}}}{\mathrm{vol}(Z^{(1)}_{0})} + \frac{1_{\{z \in Z_0^{(2)} \cap (Z_0^{(1)})^c\}}}{\mathrm{vol}(Z^{(2)}_{0})}. 
\end{align}
Plugging this into the above integral gives
\begin{align}\label{e:int_voldiff_bnd}
    & \int_{\RR^d}\left|\frac{1_{\{z \in Z^{(1)}_0\}}}{\mathrm{vol}(Z^{(1)}_{0})} - \frac{1_{\{z \in Z^{(2)}_0\}}}{\mathrm{vol}(Z^{(2)}_{0})}\right| \dint z  \nonumber \\
    &\leq \frac{|\mathrm{vol}(Z_0^{(2)}) - \mathrm{vol}(Z_0^{(1)})|}{\mathrm{vol}(Z_0^{(1)})\mathrm{vol}(Z_0^{(2)})} \mathrm{vol}(Z_0^{(1)} \cap Z_0^{(2)})  + \frac{\mathrm{vol}(Z_0^{(1)} \cap (Z^{(2)}_0)^c)}{\mathrm{vol}(Z_0^{(1)})} + \frac{\mathrm{vol}(Z_0^{(2)} \cap (Z^{(1)}_0)^c)}{\mathrm{vol}(Z_0^{(2)})} \nonumber \\
    &\leq \frac{2\mathrm{vol}(Z_0^{(1)} \cap (Z^{(2)}_0)^c)}{\mathrm{vol}(Z_0^{(1)})} + \frac{2\mathrm{vol}(Z_0^{(2)} \cap (Z^{(1)}_0)^c)}{\mathrm{vol}(Z_0^{(2)})}. 
\end{align}
Thus, we have the upper bound
\begin{align*}
    \EE \left[\int_{\mathbb{R}^d} \left|\hat{g}^{(1)}_{\gamma,n}(x) - \hat{g}^{(2)}_{\gamma,n}(x)\right| \dint x \right] 
    &\leq 2\EE\left[\frac{\mathrm{vol}(Z_0^{(1)} \cap (Z^{(2)}_0)^c)}{\mathrm{vol}(Z^{(1)}_0)} + \frac{\mathrm{vol}(Z_0^{(2)} \cap (Z_0^{(1)})^c)}{\mathrm{vol}(Z^{(2)}_0)}\right].
\end{align*}

Now consider the expectation
\begin{align*}
\EE\left[\frac{\mathrm{vol}(Z_0^{(1)} \cap (Z_0^{(2)})^c)}{\mathrm{vol}(Z^{(1)}_0)}\right].
\end{align*}
The numerator has the following upper bound: letting $\rho_1(u) := \rho_{Z_0^{(1)}}(u)$ and $\rho_2(u) := \rho_{Z_0^{(2)}}(u)$ and using the fact that for \(x \neq 0\), $x \in Z_0^{(i)}$ if and only if $\rho_i(x/|x|) \geq |x|$, we have
\begin{align*}
\mathrm{vol}(Z_0^{(1)} \cap (Z_0^{(2)})^c) &=  \int_{\RR^d} \mathbf{1}_{\{x \in Z_0^{(1)}\}}\mathbf{1}_{\{ x \notin Z_0^{(2)}\}} \dint x = \int_{\mathbb{S}^{d-1}} \int_{0}^{\infty} r^{d-1} \mathbf{1}_{\{\rho_1(u) \geq r\}}\mathbf{1}_{\{ \rho_2(u) < r\}} \dint r \dint \sigma(u) \nonumber\\
&= d^{-1}\int_{\mathbb{S}^{d-1}} \mathbf{1}_{\{\rho_1(u) \geq \rho_2(u)\}}(\rho_1(u)^d - \rho_2(u)^d) \dint \sigma(u) \nonumber \\
&= d^{-1}\int_{\mathbb{S}^{d-1}} \mathbf{1}_{\{\rho_1(u) \geq \rho_2(u)\}}(\rho_{1}(u) - \rho_{2}(u))\sum_{k=0}^{d-1}\rho_{1}(u)^{d-k-1}\rho_{2}(u)^{k} \dint \sigma(u) \nonumber \\
&\leq \int_{\mathbb{S}^{d-1}}\left(\rho_{1}(u)-\rho_{2}(u)\right)_+\rho_{1}(u)^{d-1} \dint \sigma(u).
\end{align*}
By H\"older's inequality, 
\begin{align}\label{e:vol_intersection_bound}
\int_{\mathbb{S}^{d-1}}\left(\rho_{1}(u)-\rho_{2}(u)\right)_+\rho_{1}(u)^{d-1} \dint \sigma(u) &\leq \left( \int_{\mathbb{S}^{d-1}} \left(\rho_{1}(u)-\rho_{2}(u)\right)_+^d \dint \sigma(u) \right)^{\frac{1}{d}}\left(\int_{\mathbb{S}^{d-1}} \rho_{1}(u)^{d} \dint \sigma(u)\right)^{\frac{d-1}{d}} \nonumber \\
&= \left( \int_{\mathbb{S}^{d-1}} \left(\rho_{1}(u)-\rho_{2}(u)\right)_+^d \dint \sigma(u) \right)^{\frac{1}{d}}\left(d\mathrm{vol}(Z_0^{(1)})\right)^{\frac{d-1}{d}}.
\end{align}
Then, again by H\"older's and Jensen's inequalities, 
\begin{align*}
\EE\left[\frac{\mathrm{vol}(Z_0^{(1)} \cap (Z_0^{(2)})^c)}{\mathrm{vol}(Z^{(1)}_0)}\right] &\leq d^{\frac{d-1}{d}}\EE\left[\left(\frac{\int_{\mathbb{S}^{d-1}} \left(\rho_{1}(u)-\rho_{2}(u)\right)_+^d \dint \sigma(u)}{\mathrm{vol}(Z_0^{(1)})}\right)^{\frac{1}{d}}\right] \\
&\leq d^{\frac{d-1}{d}}\EE\left[\mathrm{vol}(Z_0^{(1)})^{-\frac{1}{d-1}}\right]^{\frac{d-1}{d}}\EE\left[\int_{\mathbb{S}^{d-1}} \left(\rho_{1}(u)-\rho_{2}(u)\right)_+^d \dint \sigma(u)\right]^{\frac{1}{d}}.
\end{align*}
Then, by Proposition \ref{thm:int_bnd_p} for $p = d$ and $\gamma = 1$,

\begin{align*}
     \EE\left[\int_{\mathbb{S}^{d-1}}\left(\rho_{1}(u)-\rho_{2}(u)\right)_+^d \dint \sigma(u)\right] 
     &\leq d! C_1(\phi_1, \phi_2) W_{\mathrm{geo}}(\phi_1, \phi_2),
\end{align*}
where $C_1(\phi_1, \phi_2) =  \int_{S^{d-1}} (h(\Pi_1,u) + h(\Pi_2,u))^{-1} h(\Pi_1,u)^{-d}  \sigma(\dint u)$, \(W_{\text{geo}}(\phi_1,\phi_2) \geq \mathbb{E}[\|V_1-V_2\|]\) by the choice of the optimal coupling \(\Phi\) of \(\phi_1,\phi_2\), and \(d!=\Gamma(d+1)\).

For the other expectation, we have by Theorem 10.4.1, equations (10.4) and (10.44) in \cite{weil} that for $d \geq 2$,
\begin{align*}
\EE\left[\mathrm{vol}(Z_0^{(1)})^{-\frac{1}{d-1}}\right] &= \frac{1}{\EE[\mathrm{vol}(Z^{(1)})]}\EE[\mathrm{vol}(Z^{(1)})^{\frac{d-2}{d-1}}] \leq \EE[\mathrm{vol}(Z^{(1)})]^{-\frac{1}{d-1}} = \mathrm{vol}(\Pi_1)^{\frac{1}{d-1}},
\end{align*}
where $Z^{(1)}$ is the typical cell of $\mathcal{P}_1(1)$ and we have applied Jensen's inequality using concavity of $x \mapsto x^{\frac{d-2}{d-1}}$. 
By Theorem 2 in \cite{McMullen1991} and \eqref{eq:hZ}, and recalling that $\sigma$ is the non-normalized spherical Lebesgue measure, we have 
\begin{align}\label{e:volPi-upperbnd}
\mathrm{vol}(\Pi_1) \leq \frac{1}{d!} \left(\frac{1}{\kappa_{d-1}}\int_{\mathbb{S}^{d-1}} h(\Pi_1,u) \dint \sigma(u)\right)^d 
=
 \frac{1}{d!} \left(\frac{1}{2\kappa_{d-1}}\int_{\mathbb{S}^{d-1}} \int_{\mathbb{S}^{d-1}} |\langle u,v \rangle| \dint \sigma(u) \dint \phi_1(v)\right)^d 
 = \frac{1}{d!},    
\end{align}
and thus,
\begin{align*}
\EE\left[\frac{\mathrm{vol}(Z_0^{(1)} \cap (Z_0^{(2)})^c)}{\mathrm{vol}(Z^{(1)}_0)}\right] \leq    d^{\frac{d-1}{d}} C_1(\phi_1, \phi_2)^{\frac{1}{d}} W_{\mathrm{geo}}(\phi_1,\phi_2)^{\frac{1}{d}}.  
\end{align*}
Similarly,

\begin{align*}
 \EE\left[\frac{\mathrm{vol}(Z_0^{(2)} \cap (Z_0^{(1)})^c)}{\mathrm{vol}(Z^{(2)}_0)}\right]  \leq  d^{\frac{d-1}{d}} C_2(\phi_1, \phi_2)^{\frac{1}{d}} W_{\mathrm{geo}}(\phi_1,\phi_2)^{\frac{1}{d}},  
\end{align*}
where $C_2(\phi_1, \phi_2) := \int_{S^{d-1}} (h(\Pi_1,u) + h(\Pi_2,u))^{-1} h(\Pi_2,u)^{-d}  \sigma(\dint u)$. 

Finally, by the same argument in the proof of Theorem \ref{thm:Wass_p_bnd_new}, there is a finite constant $C(\phi_1, d)$ such that for all $\phi_2$ satisfying assumption \eqref{e:closetophi1-bnd},
\begin{align*}
    \EE \left[\int_{\mathbb{R}^d} \left|\hat{g}^{(1)}_{\gamma,n}(x) - \hat{g}^{(2)}_{\gamma,n}(x)\right| \dint x\right] &\leq C(\phi_1, d) W_{\mathrm{geo}}(\phi_1,\phi_2)^{\frac{1}{d}}.
\end{align*}
\end{proof}

The preceding proposition provides an upper bound that converges to
zero as the Wasserstein distance between the directional
distributions tends to zero. On the other hand, as $n\to\infty$, an
appropriate scaling of $\gamma$ with $n$ yields consistency of both
estimators. The next proposition provides an upper bound that
captures both regimes and also depends explicitly on the Wasserstein
distance between the directional distributions.

 Before stating the result, we give two lemmas that will be used in the proof of \Cref{prop:density_bound_2} to obtain a constant that depends only on \(\phi_1\) when \(\phi_2\) is close to \(\phi_1\), similar to the constant in \Cref{prop:density-robust}. The following notation will be used below. For a convex body \(M \subset \mathbb{R}^d\) containing \(0\), let \(R_0(M)\) denote the radius of the smallest ball centered at \(0\) that contains \(M\) and let \(r_0(M)\) denote the radius of the largest ball centered at \(0\) that is contained in \(M\).
    Let \(M^{\circ}\) denote the polar of \(M\) which is defined by
    \begin{align*}
        M^{\circ} := \{x \in \mathbb{R}^d:\langle x, y \rangle \leq 1, \, \forall y \in M\}.
    \end{align*}
    Recall that for a hyperplane \(H(u,t)\), the closed half-space containing \(0\) is denoted by \(H^-(u,t)\).
    For any \((u_1,\ldots,u_m) \in (S^{d-1})^m\), let 
    \begin{align*}
        P(u_1,\ldots, u_m):= \bigcap_{i=1}^{m} H^-(u_i,1)
    \end{align*}
    denote the intersection of the closed half-spaces associated to \(H^-(u_i,1)\), \(i=1,\dots,m\).

    \begin{lemma} \label{lem:constant_local_holder}
    Let \(\phi_0\) be an even nondegenerate probability measure on
    \(S^{d-1}\).
    Then there exist
    \(\varepsilon_0>0\), \(a_0>0\), \(0<\Delta< \infty\) and pairwise disjoint open sets
    \(U_1,\ldots,U_{2d}\subset S^{d-1}\), depending only on \(\phi_0\), such
    that the following hold:

    \begin{enumerate}
        \item For every probability measure \(\phi\) on \(S^{d-1}\) satisfying \(W_{\mathrm{geo}}(\phi,\phi_0)<\varepsilon_0\), one has \(\phi(U_i)>a_0\) for all \(i \in \{1,\ldots,2d \}\).
        \item For every \((u_1,\ldots,u_{2d}) \in U_1 \times \cdots \times U_{2d}\), \(P(u_1,\dots, u_{2d})\) is a bounded polytope, contains \(0\) in its interior, and satisfies
                \begin{align} \label{eq:polytope}
                R_0\biggl(P(u_1,\ldots,u_{2d}) \biggr) \leq \Delta.
                \end{align}
    \end{enumerate}
    \end{lemma}
    \begin{proof}
    Since \(\phi_0\) is nondegenerate, we can choose linearly independent vectors
    \(e_1,\ldots,e_d\in \operatorname{supp}\phi_0\). 
    By evenness of $\phi_0$, \(-e_1,\ldots,-e_d\in \operatorname{supp}\phi_0\) as well. 
    Set \(v_i=e_i\), \(v_{d+i}=-e_i\), \(i=1,\ldots,d\).

    We first observe that \(P(v_1,\dots,v_{2d})\) is a bounded polyhedron. 
    Indeed, note that $0$ is in the interior of \(\mathrm{conv}(v_1,\ldots,v_{2d})\) and so there exists a radius \(s>0\) such that \(sB_d \subset \mathrm{conv}(v_1,\dots,v_{2d})\).
    Since the polar of \(\mathrm{conv}(v_1,\dots,v_{2d})\) is \(P(v_1,\ldots,v_{2d})\), the latter is contained in \(s^{-1} B_d\) and is thus bounded.

    Furthermore, we use the following two facts:
    (a) the map \(L_1:(u_1,\ldots,u_{2d}) \mapsto \mathrm{conv}(u_1,\ldots,u_{2d})\) from \((S^{d-1})^{2d}\) to the space of compact subsets of \(\mathbb{R}^d\) is continuous;
    (b) the map \(L_2:\mathcal{K}_0 \to \mathbb{R}\) defined by \(L_2(K) = r_0(K)\) for all \(K \in \mathcal{K}_0\) is continuous, where \(\mathcal{K}_0\) is the space of convex bodies in \(\mathbb{R}^d\) that contain \(0\) in their interior.
    Since \(\mathrm{conv}(v_1,\ldots,v_{2d}) \in \mathcal{K}_0\),  the composition of the maps \(L_2 \circ L_1\) is well defined in a neighborhood of \((v_1,\ldots,v_{2d})\).
    Consider the inverse image \(U' = (L_2 \circ L_1)^{-1}((a'/2,3a'/2))\) restricted to the neighborhood of \((v_1,\ldots,v_{2d})\), where \(a':=r_0(\mathrm{conv}(v_1,\ldots,v_{2d}))>0\).
    Note that \(U'\) is an open set containing \((v_1,\ldots,v_{2d})\) and that \(r_0(\mathrm{conv}(u_1,\ldots,u_{2d})) > a'/2>0\) for all \((u_1,\ldots,u_{2d}) \in U'\).
    Thus, \(P(u_1,\ldots,u_{2d}) \subset (2/a')B_d\) for all \((u_1,\ldots,u_{2d}) \in U'\), since \(P(u_1,\ldots,u_{2d})\) is the polar of \(\mathrm{conv}(u_1,\ldots,u_{2d})\).

    Thus, we can choose \(r>0\) small enough such that the open balls \(U_i:=B_{\mathrm{geo}}(v_i,r)\) centered at \(v_i\) of radius \(r\) for \(i=1,\ldots,2d\),
    \begin{itemize}
        \item are pairwise disjoint with \(\prod_{i=1}^{2d} U_i \subset U'\), and
        \item  \(\phi_0(U_i)>0\) since \(v_i \in \operatorname{supp}\phi_0\) for all \(i \in\{1,\ldots, 2d\}\).
    \end{itemize}

    Since \(P(u_1,\ldots,u_{2d}) \subset (2/a')B_d\) for all \(u_i \in U_i\), \eqref{eq:polytope} is satisfied with \(\Delta=2/a'\).

    It remains to prove item 1 of the statement.
    Let \(V_i:=B_{\mathrm{geo}}(v_i,r/2)\) be the open ball in \(S^{d-1}\) of radius \(\delta:= r/2\).
    Then \(V_i\subset U_i\) and \(\mathrm{geo}(V_i,U_i^c)\ge\delta\), where $\mathrm{geo}$ is the geodesic distance defined by \eqref{eq:geo}. 
    Since
    \(v_i\in\operatorname{supp}\phi_0\),
    \[
    \alpha:=\min_{1\le i\le 2d}\phi_0(V_i)>0 .
    \]

    If \(x\in V_i\) and
    \(y\notin U_i\), then \(\mathrm{geo}(x,y)\ge\delta\). Hence, for any coupling \(\pi\) of \(\phi_0\) and \(\phi\),
    \begin{align*}
        \int \mathrm{geo}(x,y) \pi(dx,dy) 
        & \geq \delta \pi(V_i \times U_i^c)
        = \delta \bigl(\pi(V_i \times S^{d-1})- \pi(V_i \times U_i)\bigr)\\
        & \geq \delta \bigl(\phi_0(V_i) - \phi(U_i) \bigr)
    \end{align*}
    which gives
    \[
    \phi(U_i)
    \ge 
    \phi_0(V_i)
    -\frac1\delta\int \mathrm{geo}(x,y)\,\pi(dx,dy).
    \]
    Taking the infimum over all couplings gives
    \[
    \phi(U_i)\ge \phi_0(V_i)-\frac{W_{\mathrm{geo}}(\phi,\phi_0)}{\delta}
    \ge \alpha-\frac{W_{\mathrm{geo}}(\phi,\phi_0)}{\delta}.
    \]
    Now set
    \[
    \varepsilon_0:=\frac{\alpha\delta}{2},
    \qquad
    a_0:=\frac{\alpha}{2}.
    \]
    If \(W_{\mathrm{geo}}(\phi,\phi_0)<\varepsilon_0\), then \(\phi(U_i)>a_0\) for all \(i\).
    \end{proof}

\begin{lemma} \label{lem:2constant_local_holder}
    Let \(\phi_0\) be a nondegenerate even probability measure on \(S^{d-1}\), let \(\gamma>0\) and let \(p \in [1,\infty)\).
    Then, there exist constants \(\varepsilon_0:=\varepsilon_0(\phi_0)\) and \(b_p(\phi_0)\), depending only on \(\phi_0\) and \(p\), such that, for every even probability measure \(\phi\) satisfying \(W_{\mathrm{geo}}(\phi,\phi_0)< \varepsilon_0\), we have 
    \begin{align}
        \mathbb{E}[R_0(Z_0(\phi))^p] \leq \frac{b_p(\phi_0)}{\gamma^p},
    \end{align}
    where \(Z_0(\phi)\) is the zero cell of the stationary Poisson hyperplane process with parameters \(\gamma\) and \(\phi\).
\end{lemma}
\begin{proof}
    The proof is the same as that of \cite[Theorem~6.3.1]{hugPoissonHyperplaneTessellations2024} with the following choice of constants \(a\), \(\Delta\) and disjoint open sets \(U_1,\ldots,U_{2d}\) of \(S^{d-1}\) that do not depend on $\phi$.

    Taking \(U_i\), \(i =1,\ldots, 2d\), \(a:=a_0, \varepsilon_0\) and \(\Delta\) to be the disjoint open sets and constants associated to \(\phi_0\) that \Cref{lem:constant_local_holder} provides and substituting them in the proof of \cite[Theorem~6.3.1]{hugPoissonHyperplaneTessellations2024}, we obtain
    \begin{align}
        \mathbb{P}[R_0(Z_0(\phi)) \geq x] \leq 2d e^{- \gamma \frac{a}{\Delta}x}, \quad x \geq 0,
    \end{align}
    whenever \(W_{\mathrm{geo}}(\phi,\phi_0) < \varepsilon_0\).
    Thus,
    \begin{align*}
        \mathbb{E}[R_0(Z_0(\phi))^p] = \int_{0}^{\infty} \mathbb{P}[R_0(Z_0(\phi))^p > t]\, dt \leq 2 d \int_{0}^{\infty} e^{- \gamma \frac{a}{\Delta} t^{1/p}} \, dt = 2d \Gamma(p+1) \biggl(\frac{\Delta}{\gamma a} \biggr)^p.
    \end{align*}
    Since \(a, \Delta\) depend only on \(\phi_0\), the proof is complete.
\end{proof}

We now state the result.

\begin{proposition}\label{prop:density_bound_2}
    Let $\hat{f}^{(1)}_{\gamma, n}$ and $\hat{f}^{(2)}_{\gamma, n}$ be the estimators of a density $f$ as defined in Proposition \ref{prop:density-robust}. 
    Assume that the hypotheses of Proposition \ref{prop:density-robust} are satisfied.
    Further assume that $f$ is globally $L$-Lipschitz and $\mathrm{supp}(f) \subseteq [0,1]^d$, and that \(\phi_2\) satisfies
    \begin{align}\label{eq:phi_2_close_condition}
        W_{\mathrm{geo}}(\phi_1,\phi_2) 
        <  
        \min\biggl\{\varepsilon_0(\phi_1), \min_{u \in S^{d-1}} h(\Pi_1,u)\biggr\},
    \end{align}
    where \(\varepsilon_0(\phi_1)\) is the constant from \Cref{lem:2constant_local_holder}.
    Then there exist random probability densities
    $\hat{g}_{\gamma,n}^{(1)}$ and
    $\hat{g}_{\gamma,n}^{(2)}$, defined on a common probability space,
    such that, for every fixed $x\in\mathbb{R}^d$,
    \[
    \hat{g}_{\gamma,n}^{(i)}(x)
    \overset{(d)}{=}
    \hat{f}_{\gamma,n}^{(i)}(x),
    \qquad i=1,2,
    \]
    and such that
    \begin{align*}
     \EE \left[\int_{[0,1]^d}\left|\hat{g}^{(1)}_{\gamma, n}(x) - \hat{g}^{(2)}_{\gamma, n}(x)\right| \dint x\right]
    &\leq \frac{\|f\|_{\infty}^{1/2}\gamma^{d/2}\left(\mathrm{vol}(\Pi_1)^{1/2} + \mathrm{vol}(\Pi_2)^{1/2}\right)}{n^{1/2}} + \frac{L \tilde{C}(\phi_1, d)}{\gamma}W_{\mathrm{geo}}(\phi_1,\phi_2)^{1/d}, 
\end{align*}
for some constant $\tilde{C}(\phi_1,d)$ depending only on $\phi_1$ and $d$, where the expectation is taken with respect to the random tessellations and $\mathcal{D}_n$.
\end{proposition}

\begin{proof}
For each $i = 1,2$, define $\hat{g}^{(i)}_{\gamma, n}$ as in \eqref{e:hatgn}.
Fix \(x \in [0,1]^d\).
Recall $\hat{g}^{(i)}_{\gamma, n}(x)$ has the same distribution as $\hat{f}^{(i)}_{\gamma, n}(x)$ for $i = 1,2$. Now, define for $i = 1, 2$,
\[
\bar g_\gamma^{(i)}(x)
:=
\mathbb E_{\mathcal D_n}
\left[
\hat g_{\gamma,n}^{(i)}(x)
\,\middle|\,
Z_0^{(i)}
\right]
=
\mathbb E_X
\left[
\frac{
\mathbf 1_{\{X\in\gamma^{-1}Z_0^{(i)}+x\}}
}{
\operatorname{vol}(\gamma^{-1}Z_0^{(i)}+x)
}
\,\middle|\,
Z_0^{(i)}
\right].
\]
By the triangle inequality, we have
\begin{align}\label{eq:three_parts}
   \EE \left[\left|\hat{g}^{(1)}_{\gamma, n}(x) - \hat{g}^{(2)}_{\gamma, n}(x)\right|\right] 
   &\leq \EE \left[\left|\hat{g}^{(1)}_{\gamma, n}(x) - \bar{g}^{(1)}_{\gamma}(x)\right|\right] + \EE\left[\left|\bar{g}^{(1)}_{\gamma}(x) - \bar{g}_{\gamma}^{(2)}(x)\right|\right] + \EE\left[\left|\bar{g}_{\gamma}^{(2)}(x) - \hat{g}^{(2)}_{\gamma, n}(x)\right|\right].
\end{align}
We bound the first and third terms on the right-hand side of \eqref{eq:three_parts} by conditioning on $Z_0^{(i)}$:

\begin{align*}
\EE\left[\left|\bar{g}_{\gamma}^{(i)}(x) - \hat{g}^{(i)}_{\gamma, n}(x)\right|\right] &= \EE\left(\EE_{\mathcal{D}_n}\left[\left|\EE_{X}\left[\frac{1_{\{X \in \frac{1}{\gamma}Z_0^{(i)} + x\}}}{\mathrm{vol}(\frac{1}{\gamma}Z_0^{(i)} + x)}\right] - \frac{1}{n} \sum_{j=1}^n \frac{1_{\{X_j \in \gamma^{-1}Z^{(i)}_0 + x\}}}{\mathrm{vol}(\gamma^{-1}Z_0^{(i)} + x)
    }\right|\right]\right) \\
    &= \frac{\gamma^d}{n} \EE\left(\frac{1}{\mathrm{vol}(Z_0^{(i)})} \EE_{\mathcal{D}_n}\left[ \left|n \EE_{X}\left[1_{\{X \in \frac{1}{\gamma}Z_0^{(i)} + x\}}\right] - \sum_{j=1}^n 1_{\{X_j \in \gamma^{-1}Z^{(i)}_0 + x\}}\right|\right]\right). 
\end{align*}
Note that conditioned on $Z_0^{(i)}$, the independence of the data samples means that the random sum $\sum_{j=1}^n 1_{\{X_j \in \gamma^{-1}Z^{(i)}_0 + x\}}$ has the distribution of a binomial random variable with $n$ trials and probability of success $\mathbb{P}_X(X \in \gamma^{-1} Z_0^{(i)} + x )$. 
By the Cauchy--Schwarz inequality, we have
\begin{align*}
\EE_{\mathcal{D}_n}\left[ \left|n \EE_{X}\left[1_{\{X \in \frac{1}{\gamma}Z_0^{(i)} + x\}}\right] - \sum_{j=1}^n 1_{\{X_j \in \gamma^{-1}Z^{(i)}_0 + x\}}\right|\right]    &\leq \sqrt{\mathrm{Var}\left(\sum_{j=1}^n 1_{\{X_j \in \gamma^{-1}Z^{(i)}_0 + x\}}\right)} \\
&\leq \sqrt{n \mathbb{P}_X(X \in \gamma^{-1} Z_0^{(i)} + x )} \leq \sqrt{n \|f\|_{\infty} \gamma^{-d} \mathrm{vol}(Z_0^{(i)})}.
\end{align*}
Thus,
\begin{align*}
\EE\left[\left|\bar{g}_{\gamma}^{(i)}(x) - \hat{g}^{(i)}_{\gamma, n}(x)\right|\right] & \leq \sqrt{\frac{\gamma^{d}\|f\|_{\infty}}{n}}\EE\left[\mathrm{vol}(Z_0^{(i)})^{-\frac{1}{2}}\right] \leq  \sqrt{\frac{\gamma^{d}\|f\|_{\infty}}{n}}\EE\left[\mathrm{vol}(Z_0^{(i)})^{-1}\right]^{1/2} =   \sqrt{\frac{\gamma^{d}\|f\|_{\infty}\mathrm{vol}(\Pi_i)}{n}},
\end{align*}
where in the last equality we have applied Theorem 10.4.1 in \cite{weil} and \cite[(10.4) and (10.44)]{weil}.

For the second term of \eqref{eq:three_parts}, a change of variables $y = \frac{z}{\gamma} + x$ gives
\begin{align*}
 \left|\bar{g}^{(1)}_{\gamma}(x) - \bar{g}_{\gamma}^{(2)}(x)\right| 
 &= \left|\int_{\RR^d} f(y) \left(\frac{1_{\{y \in \gamma^{-1}Z^{(1)}_0 + x\}}}{\mathrm{vol}(\gamma^{-1}Z^{(1)}_{0} + x)} - \frac{1_{\{y \in \gamma^{-1}Z^{(2)}_0 + x\}}}{\mathrm{vol}(\gamma^{-1}Z^{(2)}_{0}+x)}\right)\dint y\right| \\
   &= \left|\int_{\RR^d} \left(f\left(\frac{z}{\gamma} + x\right) - f(x) \right) \left(\frac{1_{\{z \in Z^{(1)}_0\}}}{\mathrm{vol}(Z^{(1)}_{0})} - \frac{1_{\{z \in Z^{(2)}_0\}}}{\mathrm{vol}(Z^{(2)}_{0})}\right) \dint z\right| .
\end{align*}
Note here that if the density $f$ is continuous at $x$, then this term converges to zero as $\gamma \to \infty$. Under the assumption that the density is $L$-Lipschitz, we further have
\begin{align} \label{eq:two_parts}
 &|\bar{g}^{(1)}_{\gamma}(x) - \bar{g}_{\gamma}^{(2)}(x)| \leq \frac{L}{\gamma}\int_{\RR^d} \|z\| \left|\frac{1_{\{z \in Z^{(1)}_0\}}}{\mathrm{vol}(Z^{(1)}_{0})} - \frac{1_{\{z \in Z^{(2)}_0\}}}{\mathrm{vol}(Z^{(2)}_{0})}\right|\dint z  \nonumber \\
  &\leq \frac{L}{\gamma} \left(\int_{\RR^d} \|z\| \left|\frac{1_{\{z \in Z^{(1)}_0 \cap Z^{(2)}_0\}}}{\mathrm{vol}(Z^{(1)}_{0})} - \frac{1_{\{z \in Z^{(1)}_0 \cap Z^{(2)}_0\}}}{\mathrm{vol}(Z^{(2)}_{0})}\right| + \|z\| \frac{1_{\{z \in Z^{(1)}_0 \cap (Z^{(2)}_0)^c\}}}{\mathrm{vol}(Z^{(1)}_{0})} + \|z\| \frac{1_{\{z \in (Z^{(1)}_0)^c \cap Z^{(2)}_0\}}}{\mathrm{vol}(Z^{(2)}_{0})} \dint z \right) \nonumber \\
  &\leq \frac{L}{\gamma} 
  \bigg(\mathrm{diam}(Z^{(1)}_0 \cap Z^{(2)}_0)
  \mathrm{vol}\bigl(Z_0^{(1)} \cap Z_0^{(2)}\bigr)
  \left(\frac{\mathrm{vol}(Z^{(1)}_0 \cap (Z_0^{(2)})^c) + \mathrm{vol}((Z^{(1)}_0)^c \cap Z_0^{(2)})}{\mathrm{vol}(Z^{(1)}_{0})\mathrm{vol}(Z^{(2)}_{0})}\right) \nonumber \\
  & \qquad \qquad + \mathrm{diam}(Z_0^{(1)})\frac{\mathrm{vol}(Z^{(1)}_0 \cap (Z^{(2)}_0)^c)}{\mathrm{vol}(Z^{(1)}_{0})}  + \mathrm{diam}(Z_0^{(2)})\frac{\mathrm{vol}((Z^{(1)}_0)^c \cap Z^{(2)}_0)}{\mathrm{vol}(Z^{(2)}_{0})} \bigg) \nonumber \\
  &\leq \frac{2L}{\gamma}\left(\mathrm{diam}(Z_0^{(1)})\frac{\mathrm{vol}(Z^{(1)}_0 \cap (Z^{(2)}_0)^c)}{\mathrm{vol}(Z^{(1)}_{0})}  + \mathrm{diam}(Z_0^{(2)})\frac{\mathrm{vol}((Z^{(1)}_0)^c \cap Z^{(2)}_0)}{\mathrm{vol}(Z^{(2)}_{0})} \right).
  \end{align}
By \eqref{e:vol_intersection_bound} and H\"older's inequality, the expectation of the first term in the parentheses of \eqref{eq:two_parts} satisfies
\begin{align}\label{e:Ediamvolfracbnd}
 \mathbb{E}\left[\mathrm{diam}(Z_0^{(1)})\frac{\mathrm{vol}(Z^{(1)}_0 \cap (Z^{(2)}_0)^c)}{\mathrm{vol}(Z^{(1)}_{0})}\right] &\leq d^{\frac{d-1}{d}}\mathbb{E}\left[\frac{\mathrm{diam}(Z_0^{(1)})}{\mathrm{vol}(Z^{(1)}_{0})^{\frac{1}{d}}}\left( \int_{\mathbb{S}^{d-1}} \left(\rho_{1}(u)-\rho_{2}(u)\right)_+^d \dint \sigma(u) \right)^{1/d}\right] \nonumber \\
 &\leq d^{\frac{d-1}{d}}\mathbb{E}\left[\frac{\mathrm{diam}(Z_0^{(1)})^{\frac{d}{d-1}}}{\mathrm{vol}(Z^{(1)}_{0})^{\frac{1}{d-1}}}\right]^{\frac{d-1}{d}}\mathbb{E}\left[\int_{\mathbb{S}^{d-1}} \left(\rho_{1}(u)-\rho_{2}(u)\right)_+^d \dint \sigma(u)\right]^{\frac{1}{d}}.
\end{align}
By Proposition \ref{thm:int_bnd_p} for $p = d$ and $\gamma = 1$,

\begin{align*}
     \EE\left[\int_{\mathbb{S}^{d-1}}\left(\rho_{1}(u)-\rho_{2}(u)\right)_+^d \dint \sigma(u)\right] 
     &\leq \Gamma(d+1) C_1(\phi_1, \phi_2) W_{\mathrm{geo}}(\phi_1, \phi_2),
\end{align*}
where $C_1(\phi_1, \phi_2) =  \int_{S^{d-1}} \frac{1}{(h(\Pi_1,u) + h(\Pi_2,u)) h(\Pi_1,u)^d}  \sigma(\dint u)$ and \(\mathbb{E}[\|V_1-V_2\|] \leq W_{\text{geo}}(\phi_1,\phi_2)\) by the choice of the optimal coupling \((V_1,V_2) \sim \Phi\) of \(\phi_1\) and \(\phi_2\).
For the other expectation in \eqref{e:Ediamvolfracbnd}, we have by Theorem 10.4.1, equations (10.4) and (10.44) in \cite{weil} that for $d \geq 2$,

\begin{align*}    
    \EE\left[\frac{\mathrm{diam}(Z_0^{(1)})^{\frac{d}{d-1}}}{\mathrm{vol}(Z_0^{(1)})^{\frac{1}{d-1}}}\right] 
    &= \frac{1}{\EE[\mathrm{vol}(Z^{(1)})]}\EE[\mathrm{diam}(Z^{(1)})^{\frac{d}{d-1}}\mathrm{vol}(Z^{(1)})^{1-\frac{1}{d-1}}] 
    \leq (\kappa_d)^{\frac{d-2}{d-1}}\mathrm{vol}(\Pi_1)\EE[\mathrm{diam}(Z^{(1)})^{d}],
\end{align*}
where $Z^{(1)}$ is the typical cell of $\mathcal{P}_1(1)$, and the inequality follows from the fact that $\mathrm{vol}(Z^{(1)}) \leq \kappa_d \mathrm{diam}(Z^{(1)})^d$. 

Note that analogous upper bounds hold for the expectation of the second term in the parentheses of \eqref{eq:two_parts}, and thus combining all these bounds and using \eqref{e:volPi-upperbnd} gives

\begin{align}
  \EE\left|\bar{g}^{(1)}_{\gamma}(x) - \bar{g}_{\gamma}^{(2)}(x)\right| 
  &\leq \frac{2L}{\gamma} d^{\frac{d-1}{d}} \biggl(\frac{\kappa_d}{d!}\biggr)^{\frac{d-2}{d}}
   \Bigg(\EE[\mathrm{diam}(Z^{(1)})^{d}]^{\frac{d-1}{d}}C_1(\phi_1,\phi_2)^{\frac{1}{d}} \nonumber \\
  &\qquad +\EE[\mathrm{diam}(Z^{(2)})^{d}]^{\frac{d-1}{d}}C_2(\phi_1,\phi_2)^{\frac{1}{d}}\Bigg)W_{\mathrm{geo}}(\phi_1,\phi_2)^{\frac{1}{d}}. \label{eq:density_bound}
  \end{align}
  By \cite[Theorem 6.2.13]{hugPoissonHyperplaneTessellations2024}, for each \(i \in \{1,2\}\), there exists a random polytope \(Z'\) having the same distribution as the typical cell \(Z^{(i)}\) and satisfying \(Z' \subseteq Z_0^{(i)}\) almost surely, where \(Z_0^{(i)}\) is the zero cell of \(\mathcal{P}_i(1)\).
  This implies that \(\mathbb{P}[\mathrm{diam}(Z')^{\alpha} \geq t] \leq \mathbb{P}[\mathrm{diam}(Z_0^{(i)})^{\alpha} \geq t]\) for any \(t>0, \alpha>0\).
  So, \(\mathbb{E}[\mathrm{diam}(Z^{(i)})^{d}] \leq \mathbb{E}[\mathrm{diam}(Z_0^{(i)})^{d}]\) for \(i=1,2\).
  The latter expectation is bounded above by \(2^d\mathbb{E}[R_0(Z_0^{(i)})^{d}]\).
  Since \(\phi_2\) satisfies \eqref{eq:phi_2_close_condition}, by \Cref{lem:2constant_local_holder} we have \(\mathbb{E}[R_0(Z_0^{(i)})^{d}] \leq b_d(\phi_1)\) for \(i = 1,2\), where \(b_d(\phi_1)\) depends only on \(\phi_1\) and \(d\) (recall that \(Z_0^{(i)}\) are the zero cells of Poisson hyperplane processes with unit intensity).
  Therefore,
  \begin{align*}
    \mathbb{E}[\mathrm{diam}(Z^{(i)})^d] \leq 2^d b_d(\phi_1), \qquad i \in \{1,2\}.
  \end{align*}
  Since \(\phi_1\) and \(\phi_2\) satisfy \eqref{e:closetophi1-bnd}, the quantities \(C_1(\phi_1,\phi_2)\) and \(C_2(\phi_1,\phi_2)\) in \eqref{eq:density_bound} are bounded uniformly by finite constants depending only on \(\phi_1\).

  Thus, for some constant \(\tilde{C}(\phi_1,d)\) that depends only on \(\phi_1,d\), we have
  \begin{align*}
      \EE\left|\bar{g}^{(1)}_{\gamma}(x) - \bar{g}_{\gamma}^{(2)}(x)\right| & \leq \frac{L \tilde{C}(\phi_1, d)}{\gamma}W_{\mathrm{geo}}(\phi_1,\phi_2)^{\frac{1}{d}}.
  \end{align*}
Combining the above bounds, integrating over $[0,1]^d$ and applying Fubini's theorem gives the final result.

\end{proof}

Combining the preceding two propositions gives the following corollary. 
\begin{corollary}
    Let $\hat{f}^{(1)}_{\gamma, n}$ and $\hat{f}^{(2)}_{\gamma, n}$ be the estimators of a density $f$ as defined in Proposition \ref{prop:density-robust}.
    Assume that the hypotheses of Proposition \ref{prop:density-robust} are satisfied.
    Further assume that $f$ is globally $L$-Lipschitz and $\operatorname{supp}(f) \subseteq [0,1]^d$, and that \(\phi_2\) satisfies \eqref{eq:phi_2_close_condition}.
    Then there exist random probability densities
    $\hat{g}_{\gamma,n}^{(1)}$ and
    $\hat{g}_{\gamma,n}^{(2)}$, defined on a common probability space,
    such that, for every fixed $x\in\mathbb{R}^d$,
    \[
    \hat{g}_{\gamma,n}^{(i)}(x)
    \overset{(d)}{=}
    \hat{f}_{\gamma,n}^{(i)}(x),
    \qquad i=1,2,
    \]
    and such that for constants $C(\phi_1, d)$ and \(\tilde{C}(\phi_1,d)\) depending only on $\phi_1$ and $d$ we have
    \begin{align*}
    &\EE \left[\int_{[0,1]^d} \left|\hat{g}^{(1)}_{\gamma, n}(x) - \hat{g}^{(2)}_{\gamma, n}(x)\right| \dint x \right] \\
    & \leq \min\left\{\frac{\gamma^{d/2} \|f\|_{\infty}^{1/2} \left(\mathrm{vol}(\Pi_1)^{1/2} + \mathrm{vol}(\Pi_2)^{1/2}\right)}{n^{1/2}} + \frac{L \tilde{C}(\phi_1, d)}{\gamma}W_{\mathrm{geo}}(\phi_1,\phi_2)^{\frac{1}{d}}, C(\phi_1, d)W_{\mathrm{geo}}(\phi_1,\phi_2)^{\frac{1}{d}}\right\}.
    \end{align*}
\end{corollary}

\section{Conclusion}

In this work, we have studied the stability of the distribution of the zero cell of a Poisson hyperplane process with respect to perturbations of its intensity measure. In the stationary case, quantitative bounds were obtained showing local H\"older continuity with respect to perturbations of the directional distribution of the process. These results have implications for the use of random hyperplane processes in modeling and statistical applications. In particular, in this paper we illustrated how the results could be used to show stability results for stationary Poisson hyperplane or STIT density estimators as introduced in \cite{OReillyTran2021}.

There are many directions for future work that build on these results. First, it would be interesting to extend the quantitative bounds to the non-stationary setting for Poisson hyperplane tessellations that can capture local relevant features in a dataset. 
Another direction is to investigate analogous continuity properties for other geometric objects associated with hyperplane tessellations, such as the typical cell or intersections of the tessellation with lower-dimensional affine subspaces.
Finally, the techniques developed here suggest further applications to supervised learning problems. In particular, one may study regression or classification estimators based on random hyperplane or STIT tessellations \cite{OReillyTran2021minimax}, where predictions are formed by local averaging within random cells. A natural extension of the present work would be to establish stability and robustness properties for such partition-based regression estimators under perturbations of the split directions.

\section*{Acknowledgements}
The authors thank Ilya Molchanov and Bartłomiej Błaszczyszyn for their helpful suggestions concerning the continuity results for zero cells.

Eliza O'Reilly acknowledges support from NSF Grant DMS-2402234.

Bharath Roy Choudhury acknowledges support from the Dioscuri Programme under grant no.~2021/03/H/ST1/00001, initiated by the Max Planck Society, jointly managed with the National Science Centre in Poland, and mutually funded by the Polish Ministry of Education and Science and the German Federal Ministry of Education and Research.

\section*{Author information}

\textbf{Gilles Bonnet}\\
 Bernoulli Institute for Mathematics, CogniGron\\
 University of Groningen\\
  \texttt{g.f.y.bonnet@rug.nl}

  \medskip

  \noindent \textbf{Eliza O'Reilly}\\
    Department of Applied Mathematics and Statistics\\
    Johns Hopkins University\\
    \texttt{eoreilly@jhu.edu}

    \medskip

    \noindent \textbf{Bharath Roy Choudhury}\\
    Center for Advanced Mathematical Research\\
    Faculty of Mathematics and Computer Science\\
    Jagiellonian University\\
    \texttt{bharath.roy.choudhury@uj.edu.pl}

\bibliographystyle{plain}
\bibliography{biblio}

@article{McMullen1991,
author = {McMullen, Peter},
journal = {Monatshefte für Mathematik},
number = {1},
pages = {47-54},
title = {Inequalities Between Intrinsic Volumes.},
url = {http://eudml.org/doc/178509},
volume = {111},
year = {1991},
}

@inproceedings{dasgupta2008random,
  title={Random projection trees and low dimensional manifolds},
  author={Dasgupta, Sanjoy and Freund, Yoav},
  booktitle={Proceedings of the fortieth annual ACM symposium on Theory of computing},
  pages={537--546},
  year={2008}
}

@article{devroye1985nonparametric,
  title={Nonparametric density estimation},
  author={Devroye, Luc},
  journal={The L\_1 View},
  year={1985},
  publisher={John Wiley}
}

@inproceedings{devroye1983distribution,
  title={Distribution-free exponential bound on the L1 error of partitioning estimates of a regression function},
  author={Devroye, Luc and Gy{\"o}rfi, L{\'a}szl{\'o}},
  booktitle={Proceedings of the Fourth Pannonian Symposium on Mathematical Statistics},
  pages={67--76},
  year={1983},
  organization={Akad{\'e}miai Kiad{\'o} Budapest}
}

@book{gyorfi2002distribution,
  title={A distribution-free theory of nonparametric regression},
  author={Gy{\"o}rfi, L{\'a}szl{\'o} and Kohler, Michael and Krzy{\.z}ak, Adam and Walk, Harro},
  year={2002},
  publisher={Springer}
}

@book{hugPoissonHyperplaneTessellations2024,
  title = {Poisson {{Hyperplane Tessellations}}},
  author = {Hug, Daniel and Schneider, Rolf},
  year = {2024},
  series = {Springer {{Monographs}} in {{Mathematics}}},
  publisher = {Springer Nature Switzerland},
  address = {Cham},
  doi = {10.1007/978-3-031-54104-9},
  urldate = {2024-06-03},
  copyright = {https://www.springernature.com/gp/researchers/text-and-data-mining},
  isbn = {978-3-031-54103-2 978-3-031-54104-9},
  langid = {english},
}

@article{breiman2001random,
  title={Random forests},
  author={Breiman, Leo},
  journal={Machine learning},
  volume={45},
  number={1},
  pages={5--32},
  year={2001},
  publisher={Springer}
}

@inproceedings{roy2008mondrian,
  title={The {M}ondrian process},
  author={Roy, Daniel M and Teh, Yee Whye},
  booktitle={Proceedings of the 21st International Conference on Neural Information Processing Systems},
  pages={1377--1384},
  year={2008}
}

@inproceedings{lakshminarayanan2014mondrian,
  title={Mondrian forests: Efficient online random forests},
  author={Lakshminarayanan, Balaji and Roy, Daniel M and Teh, Yee Whye},
  booktitle={Advances in neural information processing systems},
  pages={3140--3148},
  year={2014}
}

@article{Nagel2005,
	Author = {Werner Nagel and Viola Weiss},
	Journal = {Advances in Applied Probability},
	Pages = {859-883},
	Title = {Crack {STIT} Tessellations: Characterization of stationary random tessellations stable with respect to iteration},
	Volume = {37},
	Year = {2005}}

@article{Thale2013Poisson,
	Author = {Tomasz Schreiber and Christoph Th{\"a}le},
	Journal = {Bernoulli},
	Number = {5A},
	Pages = {1637--1654},
	Title = {Geometry of iteration stable tessellations: Connection with {P}oisson hyperplanes},
	Volume = {19},
	Year = {2013}}

@book{weil,
	Address = {Berlin},
	Author = {Rolf Schneider and Wolfgang Weil},
	Publisher = {Springer-Verlag},
	Series = {Probability and {I}ts {A}pplications},
	Title = {Stochastic and Integral Geometry},
	Year = {2008}}

@article{Rahimi,
  title={Random Features of Large-Scale Kernel Machines},
  author={Ali Rahimi and Ben Recht},
  journal={Proceeding
NIPS'07 Proceedings of the 20th International Conference on Neural Information Processing Systems},
  volume={},
  number={},
  pages={1177-1184},
  year={2007},
  publisher={}
}

@book{MolchanovBook,
  title={Theory of Random Sets},
  author={Ilya Molchanov},
  volume={87},
  year={2017},
  publisher={Springer}
}

@article{OReillyTran2021,
  title={Stochastic Geometry to Generalize the {M}ondrian Process},
  author={Eliza O'Reilly and Ngoc Mai Tran},
  journal={SIAM Journal on Mathematics of Data Science},
  volume={4},
  number={2},
  pages={531-552},
  year={2022}
}

@book{villaniOptimalTransportOld2009,
  title = {Optimal Transport: Old and New},
  shorttitle = {Optimal Transport},
  author = {Villani, C{\'e}dric},
  year = 2009,
  series = {Grundlehren Der Mathematischen {{Wissenschaften}}},
  number = {338},
  publisher = {Springer},
  address = {Berlin Heidelberg},
  isbn = {978-3-540-71049-3 978-3-662-50180-1},
  langid = {english}
}

@article{OReillyTran2021minimax,
  title={Minimax Rates for High-dimensional Random Tessellation Forests},
  author={Eliza O'Reilly and Ngoc Mai Tran},
  journal={Journal of Machine Learning Research},
  volume={25},
  number={},
  pages={1-32},
  year={2024}
}

@book{kallenbergRandomMeasuresTheory2017,
  title = {Random {{Measures}}, {{Theory}} and {{Applications}}},
  author = {Kallenberg, Olav},
  year = 2017,
  series = {Probability {{Theory}} and {{Stochastic Modelling}}},
  volume = {77},
  publisher = {Springer International Publishing},
  address = {Cham},
  doi = {10.1007/978-3-319-41598-7},
  urldate = {2019-12-04},
  isbn = {978-3-319-41596-3 978-3-319-41598-7},
  langid = {english}
}

\end{document}